\documentclass[12pt,leqno]{amsart}
\usepackage{amssymb,amsmath,amsthm,bm}
\usepackage{graphicx}
\usepackage[dvipsnames]{xcolor}
\usepackage{color}
\usepackage[textsize=tiny]{todonotes}
\usepackage[normalem]{ulem}
\usepackage[bookmarksopen,bookmarksdepth=3,colorlinks,citecolor=red,pagebackref,hypertexnames=true]{hyperref}
\usepackage[msc-links,nobysame,non-sorted-cites, initials]{amsrefs}
\usepackage[inline]{enumitem}
\usepackage{mathtools}
\mathtoolsset{showonlyrefs}

\definecolor{indigo}{rgb}{0.4,0,0.9}

\usepackage{lipsum}
\usepackage{mathrsfs}


\DeclareFontFamily{U}{rsfs}{\skewchar\font127 }
\DeclareFontShape{U}{rsfs}{m}{n}{%
	<-6.5> rsfs5
	<6.5-8> rsfs7
	<8-> rsfs10
}{}

\usepackage[T1]{fontenc}  
\usepackage{libertine}   
\usepackage[libertine]{newtxmath}

\def\N {{\mathbb{N}}}
\def\R {{\mathbb{R}}}

\def\Z {{\mathbb{Z}}}
\def\DD {{\mathcal{D}}}
\def\RR {{\mathcal{R}}}
\def\M  {\mathrm {M}}
\def\d {\mathrm{d}}

\def\1 {{\mbox{\boldmath 1}}}

\DeclareMathOperator{\supp}{supp}
\DeclareMathOperator{\sign}{sign}
\def\ind{\cic{1}}
\newcommand{\cic}{\bm}

\newcommand{\abs}[1]{|#1|}
\newcommand{\Abs}[1]{\left|#1\right|}

\newcommand{\Norm}[2]{\left\|#1\right\|_{#2}}

\def\XXint#1#2#3{{\setbox0=\hbox{$#1{#2#3}{\int}$}
		\vcenter{\hbox{$#2#3$}}\kern-.5\wd0}}

\def \no#1#2#3 {{\bf #1} (#3), #2.}
\def \eds#1#2#3 {#1, #2, #3.}

\newcounter{counter}

\numberwithin{equation}{section}
\numberwithin{counter2}{section}
\newtheorem{proposition}[subsection]{Proposition}
\newtheorem{theorem}[counter]{Theorem}
\newtheorem{corollary}{Corollary}
\newtheorem{lemma}[subsection]{Lemma}

\theoremstyle{definition}
\newtheorem{definition}[subsection]{Definition}
\newtheorem*{remark*}{Remark}
\newtheorem*{warn*}{A word of warning}

\newtheorem{remark}[subsection]{Remark} 
\theoremstyle{plain}

\numberwithin{corollary}{counter}

\makeatletter
\DeclareRobustCommand\widecheck[1]{{\mathpalette\@widecheck{#1}}}
\def\@widecheck#1#2{%
    \setbox\z@\hbox{\m@th$#1#2$}%
    \setbox\tw@\hbox{\m@th$#1%
       \widehat{%
          \vrule\@width\z@\@height\ht\z@
          \vrule\@height\z@\@width\wd\z@}$}%
    \dp\tw@-\ht\z@
    \@tempdima\ht\z@ \advance\@tempdima2\ht\tw@ \divide\@tempdima\thr@@
    \setbox\tw@\hbox{%
       \raise\@tempdima\hbox{\scalebox{1}[-1]{\lower\@tempdima\box
\tw@}}}%
    {\ooalign{\box\tw@ \cr \box\z@}}}
\makeatother

\author[F.\ Di Plinio]{Francesco Di Plinio} 
\address[F.\ Di Plinio]{Dipartimento di Matematica e Applicazioni, Universit\`a di Napoli \\ \newline \indent Via Cintia, Monte S.\ Angelo 80126 Napoli, Italy}
\email{\href{mailto:francesco.diplinio@unina.it}{\textnormal{francesco.diplinio@unina.it}}}

\author[M.\ Fl\'orez-Amatriain]{Mikel Fl\'orez-Amatriain}
\address[M.\ Fl\'orez-Amatriain]{BCAM - Basque Center for Applied Mathematics,
48009 Bilbao, Spain}
\email{\href{mailto:mflorez@bcamath.org}{\textnormal{mflorez@bcamath.org}}}

\author[I. Parissis]{Ioannis Parissis}
\address[I.\ Parissis]{Departamento de Matem\'aticas, Universidad del Pa\'is Vasco, Aptdo. 644, 48080 Bilbao, Spain and Ikerbasque, Basque Foundation for Science, Bilbao, Spain}
\email{\href{mailto:ioannis.parissis@ehu.es}{\textnormal{ioannis.parissis@ehu.eus}}}

\author[L.\ Roncal]{Luz Roncal}
\address[L.\ Roncal]{BCAM - Basque Center for Applied Mathematics,
48009 Bilbao, Spain \\ \newline \indent  and Ikerbasque, Basque Foundation for Science, Bilbao, Spain\\ \newline \indent and Universidad del Pa\'is Vasco, Bilbao, Spain}
\email{\href{mailto:lroncal@bcamath.org}{\textnormal{lroncal@bcamath.org}}}

\date{\today}

\numberwithin{figure}{section}

\begin{document}


\thanks{F.\ Di Plinio is partially supported by the FRA 2022 Program of University of Napoli Federico II, project ReSinAPAS -
Regularity and Singularity in Analysis, PDEs, and Applied Sciences. }

\thanks{M. Fl\'orez-Amatriain is partially supported by the projects CEX2021-001142-S, PID2021-122156NB-I00/AEI/10.13039/501100011033 funded by Agencia Estatal de Investigaci\'on and acronym ``HAMIP'', grants BERC 2022-2025 of the Basque Government and predoc Basque Government grant 2022 ``Programa Predoctoral de Formaci\'on de Personal Investigador No Doctor''}

\thanks{I. Parissis is partially supported by grant PID2021-122156NB-I00 funded by MICIU/AEI/10.13039/501100011033 and FEDER, UE, grant IT1615-22 of the Basque Government and IKERBASQUE}

\thanks{L. Roncal is partially supported by the projects CEX2021-001142-S, RYC2018-025477-I, and CNS2023-143893, funded by Agencia Estatal de Investigaci\'on, and PID2023-146646NB-I00 funded by MICIU/AEI/10.13039/501100011033 and by ESF+, grant BERC 2022-2025 of the Basque Government, and IKERBASQUE}

\thanks{Declarations of interest: none.}


\title[Rubio de Francia square functions]{Pointwise localization   and sharp weighted bounds\\ for Rubio de Francia square functions}
	
\begin{abstract} Let $H_\omega f$ be the Fourier restriction of $f\in L^2(\R)$ to an interval $\omega\subset \R$.
If $\Omega$ is an arbitrary collection of pairwise disjoint intervals, the square function of $\{H_\omega f: \omega \in \Omega\}$	is termed the Rubio de Francia square function $T^{\Omega}_{\operatorname{RF}}$. This article proves a pointwise   bound for $T^{\Omega}_{\operatorname{RF}}$ by a sparse operator involving local $L^2$-averages. A pointwise bound for the smooth version of $T^{\Omega}_{\operatorname{RF}}$ by a sparse square function is also proved. These pointwise localization principles lead to quantified  $L^p(w)$, $p>2$ and weak $L^p(w)$, $p\geq 2$ norm inequalities for $T^{\Omega}_{\operatorname{RF}}$. In particular,  the obtained weak $L^p(w)$ norm bounds are new for $p\geq 2$ and sharp for $p>2$. The proofs rely on sparse bounds for abstract balayages of Carleson sequences, local orthogonality and very elementary time-frequency analysis techniques. 

The paper also contains two results related to the outstanding conjecture that $T^{\Omega}_{\operatorname{RF}}$ is bounded on $L^2(w)$ if and only if $w\in A_1$. The conjecture is verified for radially decreasing even $A_1$ weights, and in full generality for the Walsh group analogue of $T^{\Omega}_{\operatorname{RF}}$.
\end{abstract}	
	\maketitle
	
\section{Introduction and main results} \label{SIntro}

The $L^p$-norm, $1<p<\infty$, equivalence between $f$ and its Littlewood--Paley square function lies at the foundation of the modern treatment of singular integrals. The fact that this equivalence extends to weighted $L^p(w)$ norms for weights in the Muckenhoupt class testifies the localized nature of the Littlewood--Paley inequalities. In contrast to the lacunary Littlewood--Paley configuration, this article addresses the localization properties of square functions of both smooth and rough multipliers supported on frequency intervals forming an \emph{arbitrary} pairwise disjoint, or finitely overlapping, collection; precise definitions are given below.

For  intervals $\omega\subset \R$, define the class of multipliers adapted to $\omega$ as follows. Say $m\in \mathbb M_\omega$ if $m\in \mathcal C^{D}(\omega)$ for a fixed large integer $D$ and
\[
\supp m\subset \omega, \qquad \sup_{\xi \in \omega} \sup_{0\leq j \leq D} \mathrm{dist}(\xi, \partial \omega)^{j} \left\| m^{(j)} \right\|_\infty \leq 1.
\]
To a collection of pairwise disjoint intervals $\Omega$, and a choice $\{m_\omega\in \mathbb M_\omega: \omega\in  \Omega\}$, associate the square function
\[
T^\Omega f\coloneqq  \left(\sum_{\omega \in \Omega} |T_\omega f|^2 \right)^{\frac12}, \qquad T_\omega f(x)\coloneqq \int_{\mathbb R} \widehat {f}(\xi ) m_\omega (\xi) \mathrm{e}^{-i\xi x} \frac{\mathrm{d}\xi}{\sqrt{2\pi}} , \qquad x\in \mathbb R.
\]
The operator
\begin{equation}
\label{e:rc} 
H_\omega   f(x)\coloneqq \int_{\omega} \widehat {f}(\xi ) \mathrm{e}^{-i\xi x} \frac{\mathrm{d}\xi}{\sqrt{2\pi}},\qquad x\in\R,	
\end{equation}
is an instance of $T_{\omega}$ corresponding   to the choice $m_\omega=\cic{1}_\omega$. This specific case of $T^{\Omega}$ is the so-called \textit{Rubio de Francia square function}, which is assigned the notation $T^{\Omega}_{\operatorname{RF}}$
\[
T^{\Omega}_{\operatorname{RF}}f(x) \coloneqq \left(\sum_{\omega\in\Omega}|H_\omega f(x)|^2\right)^{\frac12},\qquad x\in\R.
\]
With more details and discussion to follow, one of the main results of this paper is the pointwise control of $T^{\Omega}_{\operatorname{RF}}f$ by a \emph{sparse form}, see \S\ref{sec:sparse}, in a sharp way, leading to new and in several cases best possible weighted norm inequalities for this operator.

\begin{theorem}\label{thm:prelim} Let $\Omega$ be a collection of pairwise disjoint intervals and $T^{\Omega}_{\operatorname{RF}}$ be as above. For every $f\in L^2(\R)$ with compact support there exists a sparse collection $\mathcal S$ such that
\[
T^{\Omega}_{\operatorname{RF}}f\lesssim \sum_{Q\in \mathcal S} \langle f \rangle_{2,Q}   \ind_Q 
\]
and the $L^2$-average on the right hand side cannot be replaced by any $L^p$-average for any $p<2$. Furthermore there holds
\[
 \left\| T^{\Omega}_{\operatorname{RF}} \right\|_{L^2(w)\to L^{2,\infty}(w)}\lesssim \left[  [w]_{A_{1}}[w]_{A_{\infty}}\log\left(\mathrm{e}+  [w]_{A_{\infty}} \right)\right]^{\frac12}
 \]
 and for $2<p<\infty$
 \[
\left\| T^{\Omega}_{\operatorname{RF}} \right\|_{L^p(w)\to L^{p,\infty}(w)}\lesssim [w]_{A_{\frac p2}}^{ \frac{1}{p}}[w]_{A_\infty}^{ \frac{1}{p'}}.
 \]
 The first estimate is best possible up to the logarithmic term while the second estimate is best possible.
\end{theorem}
We will get Theorem~\ref{thm:prelim} as a consequence of more general corresponding results for square functions $T^\Omega$ defined in terms of more general multipliers in $\{\mathbb M_\omega\}_{\omega\in\Omega}$, as described above; see Theorem~\ref{t:sparsergh}, Corollary~\ref{cor:roughweighted} and \S\ref{sec:sharp}.

A smooth, intrinsic counterpart of $T^\Omega$ is defined as follows. For each interval $\omega\subset \mathbb R$, let $\Phi_\omega$ be the class of functions
\[
\Phi_\omega \coloneqq \left\{
\phi \in \mathcal S(R):\, \supp \phi \subset \omega, \,\,\sup_{0 \leq j \leq D} \ell_\omega^j \left\| \phi^{(j)}  \right\|_\infty \leq 1\right\}
\]
for a positive integer $D$ which we fix to be sufficient large throughout the paper. Then the \emph{intrinsic smooth Rubio de Francia square function} is the operator
\begin{equation}
\label{e:fomeg}
G^\Omega f \coloneqq \left(\sum_{\omega\in \Omega}   f_\omega  ^2 \right)^{\frac12}, \qquad   f_\omega (x)\coloneqq \sup_{\phi \in  \Phi_\omega} \left|f*  \widehat \phi  (x) \right|, \qquad x\in \R.
\end{equation}
Both definitions   naturally extend to higher dimensions and/or parameters by considering collections of disjoint rectangles with respect to a fixed choice of a basis in $\R^n$ and defining the corresponding frequency projection operators. The two square functions $G^\Omega ,T^\Omega$ are related by vector-valued Littlewood--Paley inequalities, and their $L^p(\R)$ behavior, and in fact their $L^p(w)$-boundedness for weights $w\in A_p$ as well, $1<p<\infty$, are thus qualitatively equivalent.

The well-known result by Rubio de Francia \cite{RdF1985}  tells us that the operators $T^{\Omega}_{\operatorname{RF}}, G^\Omega$ are bounded on $ L^p(\R) $ for $ p \geq 2 $; see \cites{Journe,Lacey2007} for the higher parametric case.  Rubio de Francia's reliance on  local orthogonality in \cite{RdF1985} is embodied by the main step of his proof, namely  the sharp function pointwise inequality \begin{equation}
	\label{e:rdf} \left[G^\Omega f\right]^{\#} \leq C \sqrt{\mathrm{M}(|f|^2)}.
\end{equation}

\subsection{Pointwise sparse domination of $T^\Omega,G^\Omega$}\label{sec:sparse}
Estimate \eqref{e:rdf} also yields  $L^p(w)$-norm bounds for weights $w$ in appropriate Muckenhoupt classes. With the dual intent of    strengthening \eqref{e:rdf} and of  precisely quantifying these weighted estimates, we establish pointwise  domination principles for both $T^\Omega$ and $G^\Omega$, respectively involving the case $p=2$ of the  \emph{sparse operators}
\begin{equation}\label{e:sparsestuf}
	T_{p,\mathcal S} f\coloneqq \sum_{Q\in \mathcal S} \langle f\rangle_{p,Q} \cic{1}_Q, \qquad  G_{p,\mathcal S} f\coloneqq \left(\sum_{Q\in \mathcal S} \langle f\rangle_{p,Q}^2 \cic{1}_Q\right)^{\frac12}, \qquad 0<p<\infty
\end{equation}
associated to a \emph{sparse} collection $\mathcal S$ of intervals on the real line. The notations and definitions appearing in \eqref{e:sparsestuf} and in what follows are standard, and are recalled at the end of the introduction.

\begin{theorem}
\label{t:sparsesmth} Let $\Omega$ be a collection of pairwise disjoint intervals. 
For each $f\in L^2(\R)$ with compact support there exists a sparse collection $\mathcal S$ such that
\[
G^{\Omega} f \lesssim G_{2,\mathcal S} f\]
pointwise almost everywhere.
The implicit constant in the above inequality is absolute.
\end{theorem}

\begin{theorem}\label{t:sparsergh} Let $\Omega$ be a collection of pairwise disjoint intervals. 
For each $f\in L^2(\R)$ with compact support there exists a sparse collection $\mathcal S$ such that
\[
T^{\Omega} f \lesssim  T_{2,\mathcal S} f
\]
pointwise almost everywhere.
The implicit constant in the above inequality is absolute.
\end{theorem}
Pointwise domination of H\"older-continuous Calder\'on--Zygmund operators by the sparse operator $T_{1,\mathcal  S}$ is the keystone of Lerner's simple re-proof \cite{LerIMRN} of Hyt\"onen's $A_2$ theorem \cite{HytA2}. Since then, $T_{p,\mathcal S}$ have become ubiquitous in singular integral theory, to the point that  an exhaustive list of references is well beyond the purview of this article. On the other hand, the sparse square functions  $G_{1,\mathcal S}$, $G_{p,\mathcal S}$ have previously appeared in the context of weighted norm inequalities for square functions of Littlewood--Paley and Marcinkiewicz type, see e.g.\ \cites{Br2020,DSLR2016,LerMRL2019} and references therein. Thus, the specific relevance of the sparse domination principles of Theorems \ref{t:sparsesmth} and \ref{t:sparsergh}, beyond the strengthening of \eqref{e:rdf}, is explained by the next proposition involving weights and $A_p$ weight constants, whose    standard definitions are also recalled at the end of the introduction.

\begin{proposition} \label{p:weight} The estimates below hold with implicit constants possibly depending only on  the exponents $p,q$ appearing therein and in particular independent of the sparse collection $\mathcal S$.
\begin{itemize}\setlength\itemsep{.6em}
\item[$\mathrm{(i)}$]  $\displaystyle \left\| G_{2,\mathcal S} \right\|_{L^2(w)\to L^{2,\infty}(w)}\lesssim \left[  [w]_{A_{1}}\log\left(\mathrm{e
}+  [w]_{A_{\infty}} \right)\right]^{\frac12}$.
\item[$\mathrm{(ii)}$]  $\displaystyle \left\| T_{2,\mathcal S} \right\|_{L^2(w)\to L^{2,\infty}(w)}\lesssim \left[  [w]_{A_{1}}[w]_{A_{\infty}}\log\left(\mathrm{e}+  [w]_{A_{\infty}} \right)\right]^{\frac12}$.

\item[$\mathrm{(iii)}$]  $\displaystyle \left\| G_{2,\mathcal S} \right\|_{L^p(w)\to L^{p,\infty}(w)}\lesssim [w]_{A_{\frac p2}}^{ \frac{1}{p}}[w]_{A_\infty}^{ \frac12-\frac{1}{p}}$, $\qquad 2<p<\infty$.

\item[$\mathrm{(iv)}$]  $\displaystyle \left\| T_{2,\mathcal S} \right\|_{L^p(w)\to L^{p,\infty}(w)}\lesssim [w]_{A_{\frac p2}}^{ \frac{1}{p}}[w]_{A_\infty}^{ \frac{1}{p'}}$, $\qquad 2<p<\infty$.
\item[$\mathrm{(v)}$]  $\displaystyle \left\| G_{2,\mathcal S} \right\|_{L^p(w)}\lesssim \min\left\{[w]_{A_{\frac p2}}^{\max\left\{\frac{1}{p-2},\frac12\right\}},[w]_{A_q}^{\frac12}\right\}$,  $\qquad 2\leq 2q<p<\infty$.
\item[$\mathrm{(vi)}$]  $\displaystyle \left\| T_{2,\mathcal S} \right\|_{L^p(w)}\lesssim \min\left\{[w]_{A_{\frac p2}}^{\max\left\{\frac{1}{p-2},1\right\}},[w]_{A_q}\right\}$, $\qquad 2\leq 2q<p< \infty$.
\end{itemize}
\end{proposition}

An application of  Proposition \ref{p:weight} immediately entails two corollaries of our main results.
\setcounter{counter}{2}

\begin{corollary}
\label{cor:smoothweighted}  Estimates \emph{(i)},  \emph{(iii)} and \emph{(v)} of Proposition \ref{p:weight} hold for the intrinsic  smooth square function $G^\Omega$ in place of $G_{2,\mathcal S}$.
\end{corollary}
\setcounter{corollary}{0}
\setcounter{counter}{3}
\begin{corollary}
\label{cor:roughweighted}
 Estimates \emph{(ii)},  \emph{(iv)} and \emph{(vi)} of Proposition \ref{p:weight} hold for  $T^\Omega$ in place of $T_{2,\mathcal S}$.
\end{corollary}

\begin{proof}[Proof of Proposition~\ref{p:weight}] Points (i), (iii) and the leftmost estimates  in  (v) and (vi) are essentially special cases of previously known results.  For (i), (iii), and the leftmost estimate in (v), rely on the observation that \[
\left\| G_{2,\mathcal S} \right\|_{L^{p }(w)\to L^{p,\infty }(w)} = \left\| T_{1,\mathcal S} \right\|_{L^{\frac p2 }(w)\to L^{\frac p2,\infty }(w)}^{\frac12}, \qquad 
\left\| G_{2,\mathcal S} \right\|_{L^p(w)} = \left\| T_{1,\mathcal S} \right\|_{L^{\frac p2 }(w)}^{\frac12}\]
together with the sharp bound for the appropriate weighted norm of $T_{1,\mathcal S} $. The weak-type $L^{q}(w)$ bound for $T_{1,\mathcal S}$ was sharply quantified in \cite{LISU}*{Theorem 1.2} for $q>1$ and  in \cite{FreyNieraeth}*{Theorem 1.4}  for $q=1$, whence (iii) and (i) respectively; the latter  estimate for Calder\'on--Zygmund operators for $q=1$ is contained in \cite{LOP}. The strong type $L^{q}(w)$ bound for $T_{1,\mathcal S}$ is classical, see e.\ g.\ \cites{BFP,CUP12,FreyNieraeth,HL, LN, M}.  Finally, the leftmost estimate in    (vi) is from \cite{BFP}*{Proposition 6.4}. 

The bounds  (ii), (iv) and the rightmost estimates in (v) and (vi) seemingly do not appear in past literature. Estimates (ii) and (iv) are obtained by combining (i) and (iii), respectively, with   Corollary \ref{c:goodl} below, cf.\ Section \ref{s:cww}. This corollary  is a sparse operator version of the exponential square good-$\lambda$ of Chang, Wilson and Wolff \cite{CWW}.
 The rightmost estimate in (v) is obtained by interpolating the weak-type estimates in (iii) for $\frac p2\in(q,\infty)$. Likewise, the rightmost estimate in (vi) is obtained by interpolating the weak-type estimates of (iv) in the same open range of exponents.
\end{proof} 

{
\subsection{On the sharpness of Corollaries \ref{cor:smoothweighted}, \ref{cor:roughweighted}}\label{sec:sharp}As customary in literature, the term \emph{sharpness} of a weighted estimate in the Muckenhoupt class $A_{q}$ say, refers below to whether   the functional dependence of the estimate on the weight characteristic $[w]_{A_q}$  is best possible.

With this language, estimate (i) is sharp up to the logarithmic term. It is conceivable that the appearance of such correction is related to whether   $L^2(w)$-bounds for $G^\Omega$ hold true for all $w\in A_1$, a question that remains open at the time of  writing. For $G^\Omega$,  the  leftmost estimate in (v) is sharp for $p\geq 4$, while estimate (iii) is sharp for all $2<p<\infty$.  Analogously, it is expected that the presence of the logarithmic correction in (ii) is necessary if $L^2(w)$ fails for $T^\Omega$. At the time of writing, we can only show that (ii) is sharp up to the logarithmic term. The leftmost estimate in (vi) is sharp for $p\geq 3$ and estimate (iv) is sharp for all $p>2$. The rightmost estimates in (v) and (vi) are sharp.

The above claims are verified as follows. The claimed sharpness for  strong-type $L^p(w)$-estimates ensues by combining the  main results of  \cite{LPR} with the fact that the unweighted $L^p$ bounds for $G^\Omega$ are $O(p^{\frac12})$, and the unweighted $L^p$-bounds for $T^\Omega$ are $O(p)$ as $p\to +\infty$. Similarly, in order to verify the sharpness of weak $L^p(w)$-estimates, interpolate any two such estimates for $p$ in the open range $(2,\infty)$ with $w\in A_1$ and use \cite{LPR} again.

}

\subsection{Past literature on weighted and sparse bounds for $T^\Omega, G^\Omega$}

{In \cite[Theorem 6.1]{RdF1985}, Rubio de Francia proved that $G^{\Omega}$, and hence $T^{\Omega}_{\operatorname{RF}}$, are bounded on $L^p(w)$ for $2<p<\infty$ and $w\in A_{\frac p2}$.   
 The $L^2(w)$-bounded\-ness for $w\in A_1$ of $T^\Omega_{\operatorname{RF}}$ and $G^\Omega$, conjectured in \cite[Section 6, p. 10]{RdF1985}, see also \cite[Section 8.2, pp. 186--187]{Duobook}, remains an open question at the time of writing. Corroborating this conjecture is that it does hold for  the particular case of congruent intervals \cite[Theorem A]{RdF1983}, as well as the partial result that  $T^\Omega_{\operatorname{RF}},G^\Omega$ are $ L^2(w) $-bounded  for $w(x)=|x|^{-\alpha}\in A_1 $, $ 0<\alpha <1 $. The latter was proved by Rubio de Francia in \cite{RdF1989}, and  a different argument  was later given by  Carbery in  \cite{CARB}*{pp. 81--93}. Weighted weak-type estimates at the endpoint $p=2$ were found in \cite[Theorem B (ii)]{Krol}, yielding the weak variant of Rubio de Francia's conjecture.}
 
 {Quantitative weighted strong (for $2<p<\infty$) and weak (at $p=2$)  estimates for $T^\Omega_{\operatorname{RF}}$ were recently obtained in \cite[Corollary 1.5 and Corollary 1.6]{GRS} as  consequence of a sparse form domination \cites{BFP,CDPO} of the bilinear form for the vector-valued  version of the Rubio de Francia square function $T^\Omega_{\operatorname{RF}}$, cf. \cite[Theorem 1.3]{GRS}.  
 In comparison with the arguments of the present paper, the sparse domination proof of \cite{GRS} relied on a combination of the stopping forms techniques of \cite{CDPO} with  deeper time-frequency tools, such as  vector-valued tree estimates and size decompositions \cite{Benea}, circumventing the usual passing through the smooth operator $G^{\Omega}$.
The pointwise sparse bound of Theorem \ref{t:sparsergh} is formally stronger than  the vector-valued sparse estimate of \cite{GRS}. Furthermore, forgoing the vector-valued formalism leads to a simpler argument devoid of vector-valued time-frequency analysis.
 
In \cite{GRS}, the quantification  of the behavior of $T^\Omega_{\operatorname{RF}}$  on $L^p(w)$    is sharp for $3\le p<\infty$. 
On the other hand, the quantitative weighted weak-type estimate at the endpoint $p=2$ was of order $  [w]_{A_{1}}^{\frac12}[w]_{A_{\infty}}^{\frac12}\log\left(\mathrm{e}+  [w]_{A_{\infty}} \right)$.}
{In the present paper, the weak-type $(2,2)$ bound of Proposition \ref{p:weight} (ii) improves by a $\left[\log\left(\mathrm{e}+  [w]_{A_{\infty}} \right)\right]^{\frac12}$ term in comparison to  \cite[Corollary 1.6]{GRS}, while the weak $(p,p)$ bound, $2<p<\infty$, is sharp. }

\subsection{The strong $L^2(w)$ inequality for the Walsh model}
The Rubio de Francia square function $T^\Omega$ has an immediate Walsh group analogue. For direct comparison with the trigonometric case, the same notation is kept for corresponding operations, to the extent possible. In stark contrast with the former, we have a proof of   $L^2(w)$-boundedness for the Walsh--Rubio de Francia square function. A precise statement is in Theorem \ref{TheoremWalshRdF} below.  Albeit Theorems \ref{t:sparsesmth} and \ref{t:sparsergh} continue to hold in the  Walsh setting, here a sharp endpoint is available, and  weighted extrapolation of the $L^2(w)$ results yields better quantified weighted $L^p(w)$-bounds for the Walsh--Rubio de Francia square function than those following from the corresponding sparse domination.

Here follow the definitions relevant to Theorem \ref{TheoremWalshRdF}.  Let $\omega=[k, m)$ be an interval with $k,m \in \mathbb N$. Define the Walsh projection operator by
\[
H_{\omega} f(x) \coloneqq \sum_{n=0}^{\infty} \ind_{\omega}(n) \langle f,W_n \rangle W_n(x),\qquad x\in \mathbb T,
\]
where $ \{ W_n : n \in \mathbb N \} $ are the  characters of the Walsh group on $\mathbb T =[0,1)$, see   \eqref{e:defwf}. For a collection $\omega \in \Omega$ of pairwise disjoint intervals in $\mathbb N$   the Walsh--Rubio de Francia square function is
\[
T^{\Omega}  f(x) \coloneqq \left( \sum_{ \omega \in \Omega } \Abs{ H_{\omega} f(x) }^2 \right)^{\frac{1}{2} },\qquad x\in\mathbb T.
\]
Due to the dyadic nature of the Walsh setting, it suffices to assume  dyadic $A_p$ conditions on the weight.   The corresponding dyadic constant  will be denoted by $A_{p,\mathcal D}$.
	
\begin{theorem}\label{TheoremWalshRdF} Let $ w \in A_1 $. Then,
\[
\Norm{ T^{\Omega} f }{ L^2(w) } \lesssim [w]_{A_{1
,\mathcal D
}}^{1/2} [w]_{A_{\infty}}^{1/2} \Norm{ f }{ L^2(w) }.
\]
{Furthermore, the sharp bound
\[
\Norm{ T^{\Omega} f }{ L^p(w) } \lesssim [w]_{A_{\frac p2
,\mathcal D
}} \Norm{ f }{ L^2(w) },\qquad 2<p<\infty,
\] holds
with implicit constants depending only on $p$. }

\end{theorem}

\subsection{The strong $L^2(w)$ inequality for radially decreasing $A_1$ weights} Our final result extends the class of weights for which the $L^2(w)$-boundedness holds to even and radially decreasing $A_1$ weights in the form of the following theorem,  giving   new insight on the open question of the $L^2(w)$-bounded\-ness for $w\in A_1$ of the Rubio de Francia square function.  

\begin{theorem}\label{thm:radial} Let $w$ be an even and radially decreasing $A_1$ weight on the real line. There holds
\[
\|G^\Omega\|_{L^2(w)} \lesssim [w]_{A_1} \|f\|_{L^2(w)},\qquad \|T^\Omega \|_{L^2(w)}\lesssim [w]_{A_\infty} ^{\frac 12} [w]_{A_1} \|f\|_{L^2(w)}.
\]
\end{theorem}

The proof of Theorem \ref{thm:radial} combines local orthogonality with a stopping time argument and is presented in Section \ref{SSufficient}.  Our argument  actually yields the  conclusions of Theorem \ref{thm:radial} under the more general, albeit more technical, assumption \eqref{EqSufficient}. The latter is in general a strengthening of the $A_1$ condition, but is equivalent to $A_1$ for even, radially decreasing weights.

\subsection{Notation and generalities}\label{sec:not} 
We shall write $X\lesssim Y$ to indicate that $X\le CY$ with a positive constant $C$ independent of significant quantities and we denote $X\simeq Y$ when simultaneously $X\lesssim Y$ and $Y\lesssim X$.

The Fourier transform  obeys the normalization
\[
\widehat{f} (\xi) =\frac{1}{\sqrt{2\pi}} \int_{\R} f(x)  \mathrm{e}^{-ix\xi}\,\mathrm{d}x, \qquad \xi \in \R.
\]
Throughout the article, for $I\subset \R$ being any interval,  denote by
\[
\chi_I(x) \coloneqq  \left[1+ {\textstyle\left(\frac{|x-c_I|}{\ell_I}\right)^2 }\right]^{-1}, \qquad x\in \R,
\]
with $c_I$ and $\ell_I$ being respectively, the center and length of $I$.
For positive localized averages and for their tailed counterpart, write
\[
\langle f \rangle_{p,I}\coloneqq  {|I|^{-\frac1p}}{ \| f\cic{1}_I\|_p}, \qquad \langle f \rangle_{p,I,\dagger}\coloneqq {|I|^{-\frac1p}}{ \left\| f\chi_I^{9}\right\|_p} , \qquad 0<p<\infty.
\]
 When $p=1$,  the subscript is omitted,  simply writing $\langle f \rangle_{I}$ and  $\langle f \rangle_{I,\dagger}$ instead.
The chosen 18-th order decay  is not a relevant feature.

\subsubsection*{Sparse collections}A collection of intervals $\mathcal{S}$ is called \emph{$\eta$-sparse}  if for every  $I \in \mathcal{S} $ there exists a subset $ E_{I} \subseteq I $ such that
\[
\abs{ E_{I} } \geq \eta \abs{ I }
\]
and the collection of sets $ \left\{ E_{I} : I \in \mathcal{S} \right\} $ is pairwise disjoint.  In this article, the exact value of $\eta$ may vary at each occurrence, although there is an absolute constant $\eta_0>0$ which bounds from below each occurrence of $\eta$. In accordance,   $\eta$ is omitted when referring to \emph{$\eta$-sparse} collections.
\subsubsection*{Dyadic grids}
The standard system of shifted dyadic grids on $\mathbb R$, see e.g.\ \cite{LN}, is 
\[
\mathcal D^{j}=\left\{2^{-n}\left[ \textstyle k+\frac{(-1)^{n} j }{3}  ,\textstyle   k+1+\frac{(-1)^{n} j}{3}\right): k, n \in \mathbb Z\right\}, \qquad j=0,1,2.
\] 
The superscript $j$ in $\mathcal D^j$ is omitted whenever fixed and clear from context. If $I$ is an interval, write $\mathcal D(I)=\{J \in \mathcal D: J \subseteq I\}$. For $k\geq 0$, $j \in \mathbb Z$ and $Q\in \mathcal D$,  denote by $Q^{(k)}\in \mathcal D$ the $k$-th dyadic parent of $Q$  and define $Q^{(k,j) }\coloneqq Q^{(k)} + j \ell_{Q^{(k) }}$, which also belongs to $\mathcal D$. 
To each $Q\in \mathcal D$, associate an instance of the decomposition 
\begin{equation}\label{eq:decomp} 
\mathcal D =  \left( \bigcup_{{|j|\leq 1}} \mathcal D(Q^{(0,j)}) \right) \cup \left( \bigcup_{\substack{ k\geq 1\\ |j|\leq 1 }} \{Q^{(k,j)}\} \right)  \cup
\left( \bigcup_{\substack{ k\geq 0\\ 2\leq|j|\leq 3 }} \mathcal D(Q^{(k,j)}) \right).
\end{equation}
Equality \eqref{eq:decomp} will be used in connection with tails estimates. It can be easily obtained as a consequence of the dyadic covering
\[
5J^{(1)} \setminus 5J = J^{(1,-2)} \cup J^{(1,2)} \cup J^{(0,3\sigma )}, \qquad 
\]
holding for each $J\in \mathcal D$, with $\sigma=1$ if $J$ is a left child of  $J^{(1)}$, and $\sigma=-1$ otherwise.
 Indeed, let $Q\in \mathcal D$ and $\mathcal J$ be the maximal elements of $\mathcal D$ contained in $\mathbb R\setminus 5Q$. The elements of $\mathcal J$ partition $\mathbb R\setminus 5Q$ and one has the disjoint union
 \[
 \mathcal J =\bigcup_{k=0}^\infty \mathcal J_k, \qquad \mathcal J_k\coloneqq\left\{J \in \mathcal D: J\cap 5Q^{(k,0)}=\varnothing, J \subset5Q^{(k+1,0)}\right\}
 \]
Notice that $\mathcal J_k$ is a partition of $5Q^{(k+1,0)}\setminus 5Q^{(k,0)} $. The maximality of $\mathcal J$ and the initial observation forces the equality
\[
\mathcal J_k=\left\{ Q^{(k+1,-2)}, \, Q^{(k+1,2)},\, Q^{(k,3\sigma )}\right\}
\]
for some $\sigma\in\{-1,1\}$.
  Therefore, for  a generic $I\in \mathcal D$, there are the following possibilities.
\begin{itemize}
\item[1.] $I\subset 5Q$. Then $I\in \mathcal D(Q^{(0,j)})$ for some $|j|\leq 2$.
\item[2.] $I\cap 5Q\neq \varnothing $, $I\not\subset 5Q$. Then $I = Q^{(k,j)} $ for some $|j|\leq 2$ and $k\geq 1$.
\item[3.] $I \subset \R\setminus 5Q$. Then $I\subset J$ for some $J\in \mathcal J_k$ and $k\geq 0$, whence $I \subset Q^{(k,j)}  $ for some $k\geq 0$ and $2\leq |j|\leq 3$. 
\end{itemize}  
From here, equality \eqref{eq:decomp} is deduced.
If $Q\in \mathcal D$ is a dyadic cube, it is convenient to introduce the non-dyadic dilates of $Q$
\begin{equation}	
\label{e:ndd}
\widetilde{Q^{(k)}}\coloneqq \bigcup_{|j|\leq 2} Q^{(k,j)}, \quad k \geq 0, \qquad \mathcal R(Q)\coloneqq \{\widetilde{Q^{(k)}}:\, k\geq 0\}.
\end{equation}
{Note that $5Q^{(k)}= \widetilde{Q^{(k)}}$ and} that $\mathcal R(Q)$ is a sparse collection, two facts used in several occasions below.

A general principle is that the operators  associated to a sparse collection $\mathcal S$ may be pointwise estimated by a finite sum of operators associated to sparse collections coming from dyadic grids. 
 More precisely, the three-grid lemma \cite[Theorem 3.1]{LN} may be easily used to deduce that for each sparse collection $\mathcal S$ there exist sparse collections $\mathcal S^j\subset \mathcal D^j,$ $j=0,1,2$ such that, cf.~\eqref{e:sparsestuf},
\begin{equation}
\label{e:gotodyad}
T_{p,\mathcal S} f\lesssim \sum_{j=0,1,2} T_{p,{\mathcal S}^j} f \qquad  G_{p,\mathcal S}f\lesssim  \sum_{j=0,1,2} G_{p,{\mathcal S}^j} f ,
\end{equation} pointwise,
with   implicit constants depending on $p$ only. Any quasi-Banach function space operator norm estimate for operators \eqref{e:sparsestuf} may be thus reduced to the case where $\mathcal S$ is a subset of a dyadic grid $\mathcal D$. 
\subsubsection*{Weight characteristics} A \emph{weight} $w$ on $\R$ is a positive, locally integrable function. For $1\leq p\leq \infty$, the $A_p$ characteristic of $w$ is defined by
\[
[w]_{A_p} \coloneqq \begin{cases} \vspace{.8em}
\displaystyle \sup_{I  } \langle w \rangle_{1,I} \left( \inf_I w\right)^{-1}, & p=1,
\\ \vspace{.6em}
\displaystyle \sup_{I } \langle w \rangle_{1,I} \langle w^{-1}\rangle_{\frac{1}{p-1},I}, & 1<p<\infty,
\\\vspace{.6em}
 \displaystyle \sup_{I} \langle \mathrm{M} (w \cic{1}_I) \rangle_{1,I} \langle w   \rangle_{1,I}  ^{-1},& p=\infty,
\end{cases}
\]
where the suprema are being taken over all intervals $I\subset \R$ and $\mathrm{M}$ is the Hardy--Littlewood maximal function.
Note that our definition of $A_\infty$ coincides with that of Wilson, see e.g.\ \cites{FreyNieraeth,HP2013,WilsonBook}, and that
\[
[w]_{A_\infty}\lesssim [w]_{A_p} \leq [w]_{A_q}, \qquad 1\leq  q<p<\infty,
\]
with absolute implicit constant, see \cite{HP2013}.
  The  formal definition of the dyadic $A_p$ characteristic $[w]_{A_p,\mathcal D}$ is the same as the usual $A_p$ constant, with the supremum therein   being replaced by the supremum over all   intervals in $\mathcal D^0(\mathbb T)$, where $\mathbb T=[0,1)$.
\subsection*{Structure of the paper}

Section \ref{s:cww} introduces the sparse operators \eqref{e:sparsestuf} as special cases of balayages of Carleson sequences and contains two relevant results: a weighted exponential good-$\lambda$ inequality for balayages and a pointwise domination of balayages by a sparse operator. Sections \ref{sub:wave} and  \ref{sec:proofB} are devoted to the proofs of Theorem  \ref{t:sparsesmth} and Theorem \ref{t:sparsergh}, respectively. Theorem \ref{TheoremWalshRdF} is shown in Section \ref{SWalsh}, and  Theorem \ref{thm:radial} is proven in Section \ref{SSufficient}.

\subsection*{Acknowledgments} The authors would like to thank the anonymous referees for an expert
reading and for providing several helpful and detailed comments. The development of this article began during F.\ Di Plinio's visit to the Basque Center of Applied Mathematics (BCAM) of Bilbao, Spain. The author gratefully acknowledges BCAM's warm hospitality. The authors would also like to thank Dariusz Kosz for many helpful suggestions on the first manuscript of this paper.
\section{Balayages of Carleson sequences} \label{s:cww} Let $\mathcal D$ be a fixed dyadic grid, $\mathsf{a}=\{a_Q:Q\in \mathcal D\}$ be any sequence of complex numbers. If $\mathcal E\subset \mathcal D$, the  $\mathcal E$-\emph{balayage}  of $\mathsf{a}$ is defined by 

\begin{equation}
\label{e:balaspars}
A_{\mathcal E}[\mathsf{a}]\coloneqq \sum_{Q\in \mathcal E} |a_Q| \cic{1}_Q.   
\end{equation} 

\begin{remark} \label{r:bala} Sparse operators are special cases of \eqref{e:balaspars}. Indeed,
if $0<p<\infty$ and  $f\in L^p_{\mathrm{loc}}(\R)$,    
\begin{equation}
\label{e:sparsestuf2}
T_{p,\mathcal S}f = A_{\mathcal S} \left[\left\{\langle f\rangle_{p,Q}:Q\in \mathcal D\right\}\right], \qquad   G_{p,\mathcal S}f = \sqrt{ A_{\mathcal S}  \left[\left\{\langle f\rangle_{p,Q}^2:Q\in \mathcal D\right\}\right]}. \end{equation} 
\end{remark}

\subsection{An exponential good-$\lambda$ inequality for sparse balayages}
The next theorem is an exponential good-$\lambda$ inequality for balayages supported on sparse collections. Its corollary has been  used in the deduction of estimates (ii), (iv) of Proposition \ref{p:weight} respectively from (i), (iii) of the same proposition. For ease of notation, given a complex sequence 
$\mathsf{a}=\{a_Q:Q\in \mathcal D\}$, indicate by $ \mathsf{a}^2\coloneqq \{a_Q^2:Q\in \mathcal D\}$.

\begin{theorem} \label{sparseCWW} Let $w\in A_{\infty}$. There exist absolute constants $C,\delta>0$ such that the following holds. Let $\mathcal S\subset \mathcal D$ be a sparse collection and $\mathsf{a}=\{a_Q:Q\in \mathcal D\}$ be any  sequence. Then for all $\lambda,\gamma>0$,
\[
w\left(\left\{  A_{\mathcal S} [\mathsf{a}]>2\lambda, \,\sqrt{ A_{\mathcal S} [\mathsf{a}^2]}\leq \gamma \lambda \right\}\right) \leq C\exp\left(- \frac{\delta\gamma^2}{[w]_{A_\infty}} \right) w\left(\left\{  A_{\mathcal S} [\mathsf{a}]>\lambda \right\}\right).
\] 
\end{theorem}

\begin{corollary} \label{c:goodl} Let $0<q,s<\infty, 0< r,t\leq \infty$.  Then 
\[
\sup \left\| T_{2,\mathcal S}: L^{q,r}(w) \to L^{s,t}(w)  \right\| \lesssim [w]_{A_\infty}^{\frac12}
\sup \left\| G_{2,\mathcal S}: L^{q,r}(w) \to L^{s,t}(w)  \right\| 
\]
with the supremum taken over all not necessarily dyadic sparse collections $\mathcal S$, and  implied constant depending on $q,r,s,t$ only.
\end{corollary} 

\begin{proof}[Proof of Theorem \ref{sparseCWW} and Corollary \ref{c:goodl}] First, in view of  Remark \ref{r:bala}, Corollary \ref{c:goodl} follows from the theorem by standard good-$\lambda$ method.
To prove the theorem, by   monotone convergence, it suffices to prove the claim for finite  sparse collections $\mathcal S$ as long as the estimate obtained is uniform in $\#\mathcal S$. Denote by $\mathcal S(Q)=\{Z\in \mathcal S: Z\subseteq Q\}$ for each $Q\in \mathcal D$. 
Also denote by $F_\lambda$ and $E_\lambda$  the sets appearing respectively in the left and right hand side of the conclusion of the theorem. Under our qualitative assumptions the set $E_\lambda$ is a finite union of intervals of $\mathcal D$, whence 
$
E_\lambda=\bigcup\{R: R\in \mathcal R\}
$
and  $\mathcal R$ is the collection of  those elements of $\mathcal D$ contained in $E_\lambda$ and maximal with respect to inclusion. Pairwise disjointness of the collection $\mathcal R$ thus reduces our claim to proving
\begin{equation}
\label{e:mainspCWW}
w(F_\lambda \cap R) \leq C\exp\left(- \frac{\delta\gamma^2}{[w]_{A_\infty}} \right) w\left(R\right) , \qquad R\in \mathcal R.
\end{equation}
If $x\in F_\lambda \cap R $, then 
\[
\begin{split}
2\lambda  &<
 A_{\mathcal S} [\mathsf{a}](x) =  A_{\mathcal S(R)}  [\mathsf{a}](x) +\sum_{\substack{Z\in \mathcal S\\ Z\supseteq R^{(1)}} }  |a_Z|  
 \leq A_{ \mathcal S(R)} [\mathsf{a}](x) +\lambda  \\ & \leq \sqrt{A_{\mathcal S(R)} [\mathsf{a}^2](x) }\left( \sum_{Z\in \mathcal S(R)} 1_Z\right)^{\frac12} + \lambda \leq   \gamma \lambda  \left( \sum_{Z\in \mathcal S(R)} 1_Z\right)^{\frac12} +\lambda,
\end{split}
\]
For the second inequality on the first line we have used that $  R^{(1)}\not\subset E_\lambda$ due to maximality of $R$.
Therefore
\[
F_\lambda \cap R \subset \left\{x\in R:   \sum_{Z\in \mathcal S(R)} 1_Z>\gamma^{-2}\right\}
\]
and \eqref{e:mainspCWW} follows from the weighted John--Nirenberg inequality. The proof of the theorem is thus complete.
\end{proof}

\subsection{Subordinated Carleson sequences} 
Let $f\in L^1(\R^d)$ be a fixed function. 
Say that the sequence  $\mathsf{a}=\left\{a_I: I\in \mathcal D\right\} $ is a  \emph{Carleson sequence subordinated to $f$} 
if
\begin{equation}
\label{e:packing}\frac{1 }{|I|}
\sum_{J\in \mathcal D(I) } |J||a_J| \leq C {\langle f\rangle_{I,\dagger}}
\end{equation}
uniformly over all $I\in \mathcal D$. The least constant $C$ such that \eqref{e:packing} holds is denoted by $\|\mathsf{a}\|_{\mathcal D}$ and termed \emph{the Carleson norm} of $\mathsf{a}$. The next proposition shows that balayages of Carleson sequences subordinated to $f$ are dominated by $1$-average sparse operators applied to $f$. \begin{proposition}  \label{p:balayage}There exists an absolute constant $C$ such that the following holds. For each $f\in L^1(\R) $ with compact support  there exists a sparse collection $\mathcal S$ of intervals with the property that for all  Carleson sequences  $\mathsf a$ subordinated to $f$, there holds
\begin{equation}\label{e:spros2}
A_{\mathcal D}[\mathsf{a}] \leq C\|\mathsf{a}\|_{\mathcal D} T_{1,\mathcal S} f ,
\end{equation}
pointwise almost everywhere.
\end{proposition}
The remainder of this section is devoted to the proof of Proposition \ref{p:balayage}. Fix a compactly supported function $f$, and choose $Q\in \mathcal D$ with
$
\supp f \subset (1+3^{-1})Q$. Further, fix $\mathsf{a}=\{a_I:I\in \mathcal D\}$ subordinated to $f$, and without loss of generality assume   $a_I\geq 0$, and normalize  $\|\mathsf{a}\|_{\mathcal D}=1$. As $\mathsf{a}$ is fixed, for a generic collection $\mathcal E$, the notation $A_{\mathcal E}$ is used in the proof in place of $A_{\mathcal E} [\mathsf{a}]$. 

For the proof of \eqref{e:spros2}, note that    \eqref{eq:decomp} readily yields the splitting
\begin{equation} \label{e:splitsparse}
\begin{split}
A_{\mathcal D}   & \leq  \sum_ {|j|\leq 1}  A_ {\mathcal D(Q^{(0,j)})}    +
\sum_{\substack{ k\geq 1 \\ |j|\leq 1}} |a_{Q^{(k,j) }}|\cic{1}_{Q^{(k,j) }} +
   \sum_{\substack{ k\geq 0 \\2\leq |j| \leq 3 }} A_{\mathcal D(Q^{(k,j) })}. 
\end{split}
\end{equation} 
The proof is articulated into two constructions. The main term in \eqref{e:splitsparse} is the $|j|\leq 1$ summation while the last two summands entail error terms. We first deal with those.

\subsection{Tails}\label{sub:tails} Here we control the latter two sums on the right hand side of \eqref{e:splitsparse}. We note preliminarily that
\begin{equation}
\label{e:trivialdist}  I \in \mathcal D, \, I^{(\ell,0)} \in\left\{Q^{(k,\pm 2)},Q^{(k,\pm 3)}\right\} \implies
\mathrm{dist}(\supp f, I) \gtrsim 2^{-\ell} \ell_I\implies   \langle f\rangle_{I,\dagger} \lesssim 2^{-6\ell} \langle f \rangle_{\widetilde{Q^{(k)}}}.
\end{equation}
To obtain the last implication we used that $\|\chi^{9}_I \cic{1}_{\supp f}\|_\infty \lesssim 2^{-9\ell}  $,  $\supp f\subset \widetilde{Q^{(k)}}$, and  $|I| \gtrsim 2^{-\ell} |\widetilde{Q^{(k)}}|$. The second summand in \eqref{e:splitsparse} is estimated by the Carleson condition for a single scale,  as follows:
\begin{equation} \label{e:splitsparse1}
\begin{split}
\sum_{\substack{ k\geq 1 \\  |j|\leq 1 }} |a_{Q^{(k,j) }}|\cic{1}_{Q^{(k,j) }}
\leq \sum_{\substack{ k\geq 1 \\ |j|\leq 1 }}  \langle f\rangle_{Q^{(k,j)},\dagger} \cic{1}_{Q^{(k,j)}} \lesssim 
  T_{1,\mathcal R(Q)} f,
 \end{split}
\end{equation}
  using the notation of \S{\ref{sec:not}}, cf.\ \eqref{e:ndd} in particular. 
For the third summand in  \eqref{e:splitsparse}  we also proceed via single scale. Indeed, applying the Carleson condition  at the second step, and following with \eqref{e:trivialdist},  we have for $\tau \in\{\pm2,\pm3\}$ that
\begin{equation} \label{e:splitsparse2}
\begin{split}
A_{\mathcal D(Q^{(k,\tau) })} &=\sum_{ \ell\geq 0} \sum_{\substack{I\in \mathcal D \\ I^{(\ell,0)} =Q^{(k,\tau) }}} |a_I| \cic{1}_I
 \leq \sum_{ \ell\geq 0} \sum_{\substack{I\in \mathcal D \\ I^{(\ell,0)} =Q^{(k,\tau) }}} \langle f \rangle_{I,\dagger} \cic{1}_I\\ &
 \lesssim   \sum_{ \ell\geq 0}2^{-6\ell} \sum_{\substack{I\in \mathcal D \\ I^{(\ell,0)} =Q^{(k,\tau) }}} \langle f \rangle_{\widetilde{Q^{(k)}}} \cic{1}_I \lesssim \langle f \rangle_{\widetilde{Q^{(k)}}}  \cic{1}_{\widetilde{Q^{(k)}}}.
\end{split}
  \end{equation}
The last inequality shows that the third summation in \eqref{e:splitsparse} is also controlled by the sparse operator  $T_{1,\mathcal R(Q)}$ as on the rightmost side of \eqref{e:splitsparse1}.

\subsection{Main term} \label{ss:main2}

The main term in \eqref{e:splitsparse}
will be controlled via the intermediate estimate
\begin{equation}\label{e:sproint2}
A_{\mathcal D(I)} \lesssim T_{1,\mathcal Q(I),\dagger} f, \qquad  T_{1,\mathcal Q,\dagger} f\coloneqq \sum_{I\in \mathcal Q} \langle f \rangle_{1,I,\dagger} \cic{1}_I.
\end{equation}
for each $I\in\{Q^{(0,j)}, |j|\leq 1\}$, where $\mathcal Q(I)$ is a suitably constructed sparse collection. Then, 
\[
T_{1,\mathcal Q(I),\dagger} f\lesssim \sum_{k\geq 0}2^{-8k}\sum_{I\in \mathcal Q(I)} \langle f \rangle_{2,2^kI} \cic{1}_I
\]
so that two applications of \cite{CR}*{Theorem A}, cf.\ \cite{CR}*{Proof of Corollary A.1}, upgrade \eqref{e:sproint2} to \[
A_{\mathcal D(I)} \lesssim T_{1,\mathcal Q'(I)} f, \qquad  I\in\{Q^{(0,j)}, |j|\leq 1\}
\]
with a possibly different sparse collection $\mathcal Q'(I)$. Combining  these bounds   with  the estimates of Subsection \ref{sub:tails} completes the proof of \eqref{e:spros2}, and in turn of Proposition \ref{p:balayage}.

The proof of \eqref{e:sproint2} is a simple John--Nirenberg type iteration argument: details are as follows. For each $R\in \mathcal D(I)$, define the  collection
\[
\begin{split}
& \mathcal S(R)\coloneqq \textrm{maximal elements of }\left\{ Z\in \mathcal D(R): \sum_{\substack {W\in \mathcal D(R)\\Z\subset W  }}   a_W>4  \langle f \rangle_{R,\dagger} \right\},
\end{split}
 \]
As $\mathcal S(R)$ is a pairwise disjoint collection, an application of the Carleson condition in the last step yields the packing estimate:
\[
\sum_{Z\in \mathcal S(R)} |Z| \leq \frac{1}{4 \langle f \rangle_{R,\dagger}  } \sum_{Z\in \mathcal S(R)} \int_{Z} A_{\mathcal D(R)} \leq \frac{\|A_{\mathcal D(R)}\|_1}{4 \langle f \rangle_{R,\dagger}  } \leq \frac{|R|}{4}
\] 
while, setting $\mathcal D^\star (R)\coloneqq  \mathcal D  (R)\setminus \bigcup_{Z\in \mathcal S(R)} \mathcal D (Z)$,
\begin{equation}
\label{e:iterative2}
A_{\mathcal D(R)} \leq A_{\mathcal D^\star (R) }+ \sum_{Z\in \mathcal S(R)} A_{\mathcal D(Z)} \leq 4 \langle f \rangle_{R,\dagger} +    \sum_{Z\in \mathcal S(R)} A_{\mathcal D(Z)}.
\end{equation} 
Setting $\mathcal Q_0\coloneqq\{I\}$,  inductively define
\[ 	\begin{split}
&\mathcal  Q_{k+1}\coloneqq \bigcup_{R \in  \mathcal{Q}_{k}}\mathcal{S}(R), \quad k=0,1,\ldots,  \qquad \mathcal Q(I) \coloneqq \bigcup_{k\geq 0} \mathcal Q_{k}\end{split}  
\]
and observe that the previously obtained packing estimate ensures $\mathcal Q(I) $ is a sparse collection. Finally, iterating \eqref{e:iterative2},
\begin{equation}
\label{e:key0}
A_{\mathcal D(I)} \leq \sum_{R\in \mathcal Q}  A_{\mathcal D^\star(R)} \leq 4 T_{1,\mathcal Q(I),\dagger} f 
\end{equation}
which is the claimed \eqref{e:sproint2}.

\section{Proof of Theorem \ref{t:sparsesmth}}\label{sub:wave}
The proof of Theorem \ref{t:sparsesmth} relies upon a  suitable discretization of $G^\Omega$ into a wave packet coefficient square function. It is not difficult to show that the square sum of the wave packet coefficients of $f$ localized on a single spatial interval is a Carleson sequence subordinated to $|f|^2$ , so that the claim of the theorem readily follows from Proposition \ref{p:balayage}.

We turn to the details.   
For $j=0,1,2$ define the corresponding \emph{$j$-th tile universe} $\mathbb S^{j}\subset \mathcal D^{0}\times \mathcal D^{j}$ as the  set of  those $I\times \
\omega\in \mathcal D^{0}\times \mathcal D^{j}$  with $\ell_I \ell_\omega =1$.
The superscript $j$ is omitted whenever fixed and clear from context.
As customary, the notation $s=I_s\times \omega_s$ is employed for $s\in \mathbb S$. If $\mathbb P\subset \mathbb S$, we write 
\begin{equation}
\label{def:ti}
 \mathbb P(I) =\{s\in \mathbb P:\, I_s =I\},\qquad \mathbb P_{\subseteq }(I) =\{s\in \mathbb P:\, I_s \subseteq I\}
\end{equation}
for each interval $I\subset \R$. For our purposes, we are especially interested in   subcollections of tiles whose frequency intervals are pairwise disjoint, the precise definition being as follows. If $\Omega\subset \mathcal D$ is a collection of pairwise disjoint intervals, write
 \begin{equation}
 \label{def:somega}
 \mathbb P^{\Omega}\coloneqq \{s\in \mathbb P:\,\omega_s=\omega \textrm{ for some } \omega \in \Omega \}.
 \end{equation}
With these notations, let \ $\mathbb P^{\Omega}(I)\coloneqq \{s\in \mathbb P^{\Omega}: I_s=I \}$ for each interval $I\subset \R$.
Fix a large integer $D$. To each tile $s=I_s\times \omega_s$, recalling that $c_{\omega_s}$ denotes the center of $\omega_s$,  we associate the $L^1$-normalized wavelet class $\Psi_s$ consisting of those $\phi\in \mathcal C^\infty(\R) $ with 
\begin{equation}
\label{e:wavsupp}  \supp\widehat \phi \subset \omega_s
, \qquad  \sup_{0 \leq j\leq D} |I_s|^{1+j} \left\| \chi_{I_s}^{-D}\big(\exp(i c_{\omega_s} \cdot)\phi\big)^{(j)}\right\|_\infty \leq 1.
\end{equation}
The intrinsic wave packet coefficient of $f\in L^2(\R)$  is then defined by the maximal quantity
\begin{equation}
\label{def:sf}
s(f) \coloneqq \sup_{\phi\in \Psi_s}\left| \langle f,\phi\rangle \right|, \qquad s\in \mathbb S.
\end{equation}
The coefficients \eqref{def:sf} may be used to construct a smooth, approximately localized  analogue of the $L^2$ norm of $f$ on the torus   $I \in \mathcal D$.    Namely, if  $\Omega$ is a collection of pairwise disjoint dyadic intervals,  set
\[
\left[ f\right]_{\mathbb S^\Omega (I)} \coloneqq \left(\sum_{s\in \mathbb S^\Omega (I)} s(f)^2 \right)^{\frac12}, \qquad I \in \mathcal D.
\]
The next lemma shows that whenever $\Omega\subset \mathcal D$ is a pairwise disjoint collection, and  $f\in L^2(\R),$ the sequence $\{\left[ f\right]_{\mathbb S^\Omega(I)}^2: I \in \mathcal D\}$ is a Carleson sequence subordinated to the function $|f|^2$. 

\begin{lemma} \label{l:orthol} There holds $\displaystyle \sum_{J \in \mathcal D(I)} |J| \left[ f\right]_{\mathbb S^\Omega(J) }^2 \lesssim |I| \langle |f|^2\rangle_{I,\dagger}  $ uniformly  over $I \in \mathcal D$.
\end{lemma}

\begin{proof} Fix $I \in \mathcal D$. It suffices to show that
\begin{equation}
\label{e:interorth}
\sum_{J \in \mathcal D(I)} \sum_{s\in {\mathbb S}^\Omega(J)} |J| |\langle f,\phi_s\rangle|^2 \lesssim |I| \langle |f|^2 \rangle_{I,\dagger} 
\end{equation}for an arbitrary choice of $\phi_s\in \Psi_s$, for 
each $J \in \mathcal D(I)$ and $s\in \mathbb{S}^{\Omega} (J)$. Set $\varphi_s \coloneqq \chi_{I}^{-9} \phi_s$. Due to localization and to  the pairwise disjoint nature of the collection $\Omega$,  the almost orthogonality estimate
\[
\left|\langle \varphi_{s}, \varphi_{s'} \rangle \right| \begin{cases}
 =0, & \omega_s\neq \omega_{s'},
\\
 \lesssim |I_s|^{-1} \mathrm{dist}(I_s,I_{s'})^{-100},  & \omega_s=\omega_{s'},
\end{cases}
\] 
holds for all $s,s'\in \mathbb{S}^\Omega $ with $I_{s},I_{s'}\in \mathcal D(I)$. A standard $TT^*$ type argument, see for example \cite{APHPR}*{\S4.3}, yields the almost orthogonality bound
\begin{equation}\label{e:interorth2}
\sum_{J \in \mathcal D(I)} \sum_{s\in {\mathbb S}^\Omega(J)} |J| |\langle g,\varphi_s\rangle|^2 \lesssim \| g\|_2^2
\end{equation}
and \eqref{e:interorth} follows by applying \eqref{e:interorth2} to $g=f\chi_I^{9}$ and relying on  the definition of $\langle\cdot \rangle_{I,\dagger}$.
\end{proof}

A combination of Proposition \ref{p:balayage} and Lemma \ref{l:orthol} immediately yields a sparse domination result for the intrinsic wave packet square function
\begin{equation}
\label{eq:intrinsic}
W^{\Omega} f \coloneqq \left( \sum_{s\in \mathbb S^\Omega} s(f)^2 \cic{1}_{I_s}\right)^{\frac 12}= \sqrt{
A_{\mathcal D} \left[\left\{[f]_{\mathbb S^\Omega(I)} ^2:I\in \mathcal D\right\}\right]}.
\end{equation}

\begin{proposition} \label{p:tfasparse} Let $f\in L^2(\R)$ be a compactly supported function and     $\Omega \subset \mathcal D$ be a pairwise disjoint collection. Then there exists a sparse collection $\mathcal S$ with the property that
\[
W^{\Omega} f \lesssim  G_{2,\mathcal S} f
\]
pointwise almost everywhere. The implicit constant in the above inequality is absolute.
\end{proposition}
Indeed, by Lemma \ref{l:orthol}, $\{[ f]_{\mathbb S^\Omega(I)}^2:\, I \in \mathcal D\}$ is a Carleson sequence subordinated to the function $|f|^2$. Thus Proposition \ref{p:tfasparse} is obtained via an application of Proposition  \ref{p:balayage}, followed by the observation that $\sqrt{T_{1, \mathcal S} \left(|f|^2\right) } = G_{2,\mathcal S} f$.

\subsection{Sparse estimates for smooth square functions: proof of Theorem~\ref{t:sparsesmth}} The relation of $G^{\Omega}$ with the wave packet square function $W^{\Omega}$ defined above is given by the pointwise estimate
\begin{equation}\label{eq:intrinstowp}
G^\Omega f \lesssim \sup_{1\leq k\leq 9} W^{\Omega^{k,\star}} f 
\end{equation}
where each $\Omega^{k,\star}$, $1\leq k \leq 9$, is a collection of pairwise disjoint intervals contained in one of the three grids $\mathcal D^{j}$, $j=0,1,2$. To obtain this pointwise bound, associate to each $\omega\in \Omega$ an index $j\in\{0,1,2\}$ and a  smoothing interval $\omega^\star \in  \mathcal D^{j}$, that is the unique interval of $\mathcal D^{j}$ with $\omega \subset  \omega^\star $ and $3 \ell_{\omega}\leq   \ell_{\omega^\star} < 6 \ell_{\omega}$. As the intervals of $\Omega$ are pairwise disjoint, $\Omega$   can be split into collections $\Omega^{k}$, $1\leq k \leq 9$, with the property that $\Omega^{k,\star}\coloneqq \{\omega^\star:\omega \in \Omega^{k}\}\subset \mathcal D^{j}$ for some $j$  and is a pairwise disjoint collection. A standard discretization procedure, see for example \cite{APHPR}*{Lemma 5.9}, then entails 
\[
G^{\Omega^k} f \lesssim W^{\Omega^{k,\star}} f 
\]
and \eqref{eq:intrinstowp} follows. We may finally combine Proposition \ref{p:tfasparse} with \eqref{eq:intrinstowp} to conclude Theorem~\ref{t:sparsesmth},  {using also that the union of $9$ sparse collections is still a sparse collection, see e.\ g.~\cite{LN}.}

\section{Proof of Theorem \ref{t:sparsergh}}
\label{sec:proofB}
The proof of Theorem \ref{t:sparsergh}, finalized at the end of this section, rests on a well known Littlewood--Paley type reduction to a model time-frequency square function appearing on the left hand side of \eqref{e:spro},  which we now introduce.
\subsection{Time-frequency square function} \label{sec:tfsf}

Fix a dyadic grid $\mathcal D$. Given an interval $\omega$ we let $\boldsymbol{\omega}\subset \mathcal D$ be a collection of dyadic intervals with the following properties.
\begin{itemize}
	\item[(i)] The collection $\boldsymbol{\omega}$ is pairwise disjoint.
	\item[(ii)] For each $k\in\Z$ there exists at most one $\alpha\in\boldsymbol{\omega}$ with $\ell_{\alpha}=2^k$.
	\item[(iii)] Each $\alpha\in \boldsymbol{\omega}$ satisfies $7^3\alpha\subset \omega$ and $7^4 \alpha\not\subset\omega$.
\end{itemize}
Observe that by (iii) the collection ${\boldsymbol{\omega}}$ is a subcollection of a dyadic Whitney covering of $\omega$ but not necessarily the whole Whitney cover. As a result properties (i) and (ii) can always be achieved by splitting $\omega$ into finitely many subcollections. Let $I\subset \R$ be any interval, possibly unbounded. 
Recalling  the definitions \eqref{def:ti}, \eqref{def:somega}, \eqref{e:wavsupp},    let $\Phi_{ \mathbb S}\coloneqq\{\varphi_s\in \Psi_s, \vartheta_s \in |I_s|\Psi_s: s \in \mathbb S\}$ be a choice of wave packets. We say that 
\begin{equation}
\label{e:tfp}
P_{I,\omega,\boldsymbol{\omega}}^{\Phi_{ \mathbb S}} f \coloneqq \sum_{s\in \mathbb S_\subseteq ^{\boldsymbol{\omega}} {(I)}} \langle f, \varphi_s\rangle\vartheta_s, \qquad \omega \in \Omega
\end{equation}
is a \emph{time-frequency projection} of $f$ on the time-frequency region $I\times \omega$. Note the $L^1$, $L^\infty$ normalizations of $\varphi_s,\vartheta_s$ respectively. We will drop the subindex $\mathbb S$ from $\Phi_{ \mathbb S}$ for the rest of the section and, whenever the choices of  $\Phi$  and $\boldsymbol \omega$ are fixed and clear from context, we will simplify the notation by writing  $P_{I,\omega}$, suppressing the dependence on $\Phi_{\mathbb S}$ and $\boldsymbol \omega$. It is easy to check that 
\begin{equation}
\label{e:tildeP}
\widetilde {P_{I,\omega}} f \coloneqq P_{I,\omega}(\chi^{-9}_I f) 
\end{equation} is a standard  Calder\'on--Zygmund operator, whence the estimates
\begin{equation}
\label{e:czo} \| P_{I,\omega}f\|_{p} \lesssim pp' |I|^{\frac1p} \langle f \rangle_{p,I,\dagger}, \qquad 1<p<\infty,
\end{equation}
hold uniformly over all bounded intervals $I$, which we will only use for $p=2$.  Furthermore, due to the frequency localization of the $\Psi_s$ classes for $s\in \mathbb S^{{\boldsymbol{\omega}}}$,   the equality
\begin{equation}
\label{e:frloc}  
P_{I,\omega} f= P_{I,\omega} H_\omega f,  
\end{equation}
holds for the frequency projection $H_{\omega}$ defined in \eqref{e:rc}.
The next theorem is a sparse domination principle for the square function   $\| P_{\R,\omega}\|_{\ell^2(\omega \in \Omega)}$ under pairwise disjointness assumption of the corresponding collection $\omega \in \Omega$. \begin{proposition}
\label{p:tfasparser} Let $\Phi$ be a choice of wave packets, $f\in L^2(\R)$ be a compactly supported function  and  $\Omega$ be a qualitatively finite, pairwise disjoint collection of intervals. 
Then there exists a sparse collection $\mathcal S$, depending on $\Phi,f,\Omega$ only  with the property that
\begin{equation}
\label{e:spro}
\left\|P_{\R,\omega}^{\Phi} f\right\|_{\ell^2(\omega\in \Omega)} \lesssim T_{2,\mathcal S} f
\end{equation}
pointwise almost everywhere, with implicit absolute numerical constant.
\end{proposition}
{The proof of  the  proposition is given in Subsection \ref{ss:pftfa}. It relies on two lemmas which we state  now. The first handles domination of tails.
\begin{lemma}
\label{l:taildomlem} Let $J\in \mathcal D$ and $\M_2 f\coloneqq (\M|f|^2)^{\frac12}$. Let  $\Phi$ be a choice of wave packets. The following pointwise bounds hold.
\begin{itemize}\setlength\itemsep{.6em}
\item[\emph{(i)}]  $\displaystyle \left\|\sum_{s\in{ \mathbb S}^{\boldsymbol{\omega}} (J)} \langle f, \varphi_s \rangle\vartheta_s \right\|_{\ell^2(\omega\in \Omega)} \lesssim \chi_{J}^{9}\langle f \rangle_{2,J,\dagger}$.
\item[\emph{(ii)}] Suppose $\mathrm{dist}(x,J) \gtrsim \ell_J$. Then $\displaystyle \left\| P_{J,\omega} f(x) \right\|_{\ell^2(\omega\in \Omega)} \lesssim \chi_{J}^{6} \mathrm{M}_2 f(x)$.
\item[\emph{(iii)}] If $  \supp f\subset 2J^{(0,\pm 2)} $, then $\displaystyle  
\left\| P_{J,\omega} f \right\|_{\ell^2(\omega\in \Omega)} \lesssim \chi_{J}^{9}\langle f \rangle_{2,7J}.$
\end{itemize}
\end{lemma}
\begin{proof} The first estimate follows immediately from the two controls
\[
\sqrt{ \sum_{\omega\in \Omega}\sum_{s\in{ \mathbb S}^{\boldsymbol{\omega}} (J)} |\langle f, \varphi_s \rangle|^2 }\lesssim \langle f \rangle_{2,J,\dagger}, \qquad \sup_{s\in{ \mathbb S}^{\boldsymbol{\omega}} (J)} |\vartheta_s| \lesssim \chi_{J}^{9},
\]
the second meant pointwise. To obtain the bound in (ii), 
\begin{equation} \label{e:taildom2}
\begin{split} & \quad 
\left\| P_{J,\omega} f(x) \right\|_{\ell^2(\omega\in \Omega)} \leq \sum_{\ell\geq 0} \sum_{\substack{I\in \mathcal D(J)\\ I^{(\ell,0)}=J}} \left\|\sum_{s\in{ \mathbb S}^{\boldsymbol{\omega}} (I)} \langle f, \varphi_s \rangle\vartheta_s (x)\right\|_{\ell^2(\omega\in \Omega)}
\\ &\lesssim \sum_{\ell\geq 0} \sum_{\substack{I\in \mathcal D(J)\\ I^{(\ell,0)}=J}} \chi_{I}^{9}(x)\langle f\rangle_{2,I,\dagger}
 \lesssim 
\mathrm{M}_2 f(x)\sum_{\ell\geq 0}   \sum_{\substack{I\in \mathcal D(J)\\ I^{(\ell,0)}=J}}\chi_{I}^{8}(x) \lesssim \chi_{J}^{6}\mathrm{M}_2 f(x),
\end{split}
\end{equation}
having applied (i) with $J=I$ for each $I$ such that $I^{(\ell,0)}=J$.
We have employed the easily verified inequalities 
\[
\chi_I(x)\langle f\rangle_{2,I,\dagger} \lesssim   \mathrm{M}_2 f(x), \qquad \sum_{\substack{I\in \mathcal D(J)\\ I^{(\ell,0)}=J}} \chi_{I}^{8}(x) \lesssim 2^{-\ell} \chi_J^6(x), \qquad 
\]
valid for $I\subset J$, $\mathrm{dist}(x,J)\gtrsim \ell_J$. To obtain the bound of (iii), start again from the   right hand side of the first line of \eqref{e:taildom2}, and apply (i) for each $I$ in the summation, so that
\begin{equation} \label{e:taildom3}
\begin{split} & \quad 
\left\| P_{J,\omega} f \right\|_{\ell^2(\omega\in \Omega)} \lesssim \sum_{\ell\geq 0} \sum_{\substack{I\in \mathcal D(J)\\ I^{(\ell,0)}=J}} \ \chi_{I}^{9} \langle f\rangle_{2,I,\dagger}
  \lesssim \chi_{J}^{9}\langle f\rangle_{2,7J}\end{split}
\end{equation}
as claimed. We have used that for each $I$ as above, $\mathrm{dist}(\supp f,I) \gtrsim 2^{\ell} \ell_I$. Together with $\supp f \subset 2 J^{(0,\pm 2)}\subset 7J$, it follows that $\langle f\rangle_{2,I,\dagger}\lesssim 2^{-8\ell}\langle f\rangle_{2,7J}$, whence the last inequality in \eqref{e:taildom3}. This completes the proof of the Lemma. 
\end{proof}
The second lemma encapsulates the main iteration of the proof of  Proposition \ref{p:tfasparser}. 
\begin{lemma}\label{l:mainiter}	 Let $J\in \mathcal D$ and $f\in L^2(\R)$. Let $\Phi$ be a choice of wave packets. Then there exists a sparse collection $\mathcal Q=\mathcal Q(\Phi,J,f,\Omega)$ with the property that, pointwise almost everywhere
\begin{equation}
	\label{e:spro2}
\left\|1_J P_{3J,\omega}^{\Phi} f \right\|_{\ell^2(\omega \in \Omega)} \lesssim T_{2,\mathcal Q} f.
\end{equation}
\end{lemma}
The proof of this lemma is more involved and thus occupies its own   subsection.
}

\subsection{Proof of Lemma \ref{l:mainiter}} The collection $\Phi$ is fixed throughout this proof and thus omitted from the notation.
 Arguing as in Subsection \ref{ss:main2}, cf.\ \eqref{e:sproint2}, it suffices to prove   the weaker result that for some sparse collection $\mathcal Q$ 
\begin{equation}\label{e:sproint}
\left\|1_J P_{3J,\omega} f \right\|_{\ell^2(\omega \in \Omega)} \lesssim T_{2,\mathcal Q,\dagger} f, \qquad  T_{2,\mathcal Q,\dagger} f\coloneqq \sum_{I\in \mathcal Q} \langle f \rangle_{2,I,\dagger} \cic{1}_I,
\end{equation}
and later upgrade  \eqref{e:sproint}  to \eqref{e:spro2} via   \cite{CR}, with a possibly different sparse collection $\mathcal Q$.

The proof of \eqref{e:sproint} rests on an iterative inequality whose first step is a stopping construction. Fix again a large constant $\Theta$ to be determined. For each $I\in \mathcal D(J)$, define   the stopping  sets and collections
 \[
 \begin{split}
&E_1(I)\coloneqq   \left\{ x\in { I}: \,\left\|\mathrm{M} P_{3I,\omega} f\right\|_{\ell^2(\omega\in \Omega)} >\Theta \langle f  \rangle_{2,I,\dagger}\right\}, 
\\[.5em] 
&E_2(I)\coloneqq   \left\{ x\in { I}: \,\left\|\mathrm{M} \left[H_\omega (f\chi_I^9)\right]\right\|_{\ell^2(\omega\in \Omega)} >\Theta \langle f  \rangle_{2,I,\dagger}\right\},
\\[.5em] 
& \mathcal S(I) \coloneqq \left\{\textrm{maximal elements } Z\in \mathcal D:\,  { Z}  \subset  E(I)\coloneqq E_1(I)\cup E_2(I)\right\}.
\end{split}
 \]
This time, the maximality condition ensures   
 \begin{align}\label{e:stop2} 
 \inf_{ { 3Z} } \left\|\mathrm{M} P_{3I,\omega} f\right\|_{\ell^2(\omega\in  \Omega)}  +   \inf_{  { 3Z}} \left\|\mathrm{M} \left[H_\omega(f\chi_I^9)\right]\right\|_{\ell^2( \omega\in \Omega)} &\lesssim \langle f  \rangle_{2,I,\dagger}, \quad Z\in \mathcal S(I),
 \\
 \label{e:stopinf} \left\|\mathrm{M} P_{3I,\omega} f(x)\right\|_{\ell^2(\omega\in \Omega)} &\leq \Theta   \langle f  \rangle_{2,I,\dagger}, \quad x \in{ { I}\setminus E(I)}.
\end{align}
Now set $\mathcal Q_0\coloneqq\{J\}$. Proceed  inductively, defining  
\[ 	\begin{split}
&\mathcal  Q_{k+1}\coloneqq \bigcup_{I \in  \mathcal{Q}_{k}}\mathcal{S}(I), \quad k=0,1,\ldots,  \qquad \mathcal Q\coloneqq \bigcup_{k\geq 0} \mathcal Q_{k}.\end{split}  
\]
 Arguing in the same way as \cite{CCDPO17}*{Proof of eq.\ (2.22)}, see also \cite{CDPPV22}*{Section 4}, the fact that $\mathcal Q$ is a sparse collection is easily verified once the estimates $|E_j(I)|<2^{-16}|I|$, $j=1,2$ are proved. In the case of $E_1(I)$, provided $\Theta $ is large enough, this  follows from Chebyshev and 
 \[
 \begin{split} \frac{1}{|I|} \int \left\|\mathrm M P_{3I,\omega} f\right\|^2_{\ell^2(\omega\in \Omega)} &\lesssim \frac{1}{|I|} \sum_{\omega\in \Omega }\ \left\|\widetilde{P_{3I,\omega}} H_\omega (f\chi_{I}^9)\right\|^2_{2} \lesssim 
 \frac{1}{|I|} \sum_{\omega\in \Omega }\ \left\|  H_\omega (f\chi_{I}^9)\right\|^2_{2}  \leq \langle f  \rangle_{2,I,\dagger}^2
 \end{split}
 \]
 having used  the maximal theorem in the first step, \eqref{e:czo} for the second inequality and orthogonality of the projections $H_\omega$ in the last.  A shorter computation leads to the same estimate for $E_2(I)$.
 The next lemma is the main device that controls the oscillation. The maximal frequency truncation idea dates back to the single tree estimate in  Lacey and Thiele's seminal paper on the Carleson operator \cite{LT}.
The proof is given at the end of this subsection.
\begin{lemma} \label{l:osclemma} Let $Z\in\mathcal S(I)$. There holds
\[  \sup_{ {Z}}  \left|P_{{3Z},\omega}f - P_{3I,\omega} f\right| \lesssim \inf_{{3Z}} \mathrm{M} P_{3I,\omega} f +  \inf_{{ 3Z}} \mathrm{M}\left[ H_{\omega} (f\chi_{I}^9)\right].
\] 
\end{lemma}
{Now for each $I\in \mathcal D(J)$, with the stopping collection $\mathcal S(I)$ at hand, pick $x\in I$. Then either $x \in I \setminus E(I)$, in which case
\begin{equation}
\label{e:goodset}	
\left\|  {P}_{3I,\omega} f(x)\right\|_{\ell^2(\omega\in \Omega)} \leq C \langle f  \rangle_{2,I,\dagger} ,
\end{equation}
by virtue of  \eqref{e:stopinf}, or $x\in Z$ for some $Z\in \mathcal S(I)$, in which case 
 \begin{equation}
\label{e:badset}	
\left\|  {P}_{3I,\omega} f(x)\right\|_{\ell^2(\omega\in \Omega)} \leq    \left\| P_{3Z,\omega} f  (x)\right\|_{\ell^2(\omega\in \Omega)} + C \langle f  \rangle_{2,I,\dagger} 
\end{equation}
via an application of Lemma \ref{l:osclemma} and \eqref{e:stop2}.
}
It follows that
\begin{equation}
\label{e:recrough} \cic{1}_{I}
\left\|  {P}_{3I,\omega} f \right\|_{\ell^2(\omega\in \Omega)} \leq  C \langle f  \rangle_{2,I,\dagger} \cic{1}_{I} +
\sum_{Z\in \mathcal Q } \cic{1}_{Z} \left\| P_{3Z,\omega} f \right\|_{\ell^2(\omega\in \Omega)}.  \end{equation}
Starting from $I=J,$ iterate \eqref{e:recrough} to obtain
\[
\cic{1}_{J} 
\left\| P_{3J,\omega} f\right\|_{\ell^2(\omega\in \Omega)}\ \lesssim \sum_{I\in \mathcal Q}  \langle f  \rangle_{2,I,\dagger} \cic{1}_I,
\]
completing the proof of \eqref{e:sproint}, and in turn of Lemma \ref{l:mainiter}.

\begin{proof}[Proof of Lemma \ref{l:osclemma}] Note that 
\begin{equation}
\label{e:lemma10}
\begin{split} & \quad 
|P_{{3Z},\omega}f - P_{3I,\omega} f|\\ &= \left|\sum_{s\in \mathbb S_\subseteq ^{{\boldsymbol{\omega}}} ({3I})\setminus \mathbb S_\subseteq ^{{\boldsymbol{\omega}}} (3Z)  } \langle f,\varphi_s\rangle \vartheta_s\right| 
\leq  \left| \sum_{\substack{s\in \mathbb S_\subseteq ^{{\boldsymbol{\omega}}} (3I)\\ \ell_{I_s}> \ell_Z }} \langle f,\varphi_s\rangle\vartheta_s \right| 
+ \left|  \sum_{\substack{|j|\geq 2\\ Z^{(0,j)}\subset 3 I }}P_{ Z^{(0,j)} ,\omega} f \right|.
\end{split}
\end{equation}
Let us handle the tail term in \eqref{e:lemma10}. {The separation between $Z$ and the small scales contained in $Z^{(0,j)}$ for some $j\geq 2$ allows for the standard Calder\'on--Zygmund tail estimate
\begin{equation}
	\label{e:retail2}
\cic{1}_{Z}
\left| \sum_{\substack{|j|\geq 2\\ Z^{(0,j)}\subset 3I }}P_{ Z^{(0,j)} ,\omega} f\right|= \cic{1}_{Z}
 \left| \sum_{\substack{|j|\geq 2\\ Z^{(0,j)}\subset 3I }} \sum_{s\in \mathbb S_\subseteq ^{{\boldsymbol{\omega}}} (Z^{(0,j)})} \langle H_\omega(f\chi_{I}^{9}) , \varphi_s \rangle \vartheta_s \right| \lesssim  \inf_{{ 3Z}} \mathrm{M}\left[ H_{\omega} (f\chi_{I}^9)\right].
\end{equation}
The proof is essentially a repetition of the one for Lemma \ref{l:taildomlem} (ii) using $L^1$ averages instead, and thus the details are omitted.}
The first term in \eqref{e:lemma10} is the main term. To handle it define the sets
\[
\beta_Z\coloneqq \mathrm{Conv}\left(  \bigcup \left\{\alpha \in \boldsymbol{\omega}:\, \ell_\alpha < (\ell_Z)^{-1}\right\}\right), \qquad  \gamma_Z\coloneqq \mathrm{Conv}\left(\bigcup \left\{\alpha \in \boldsymbol{\omega}:\, \ell_\alpha \leq  (\ell_Z)^{-1}\right\}\right).
\]
Using the Whitney property of $\boldsymbol{\omega}$, there exist  positive constants $c_1,c_2, c_3,c_4$ with $c_2-c_1 \simeq 1, c_4-c_3 \simeq  1 $
such that if ${\omega}=[a,b)$,
\[
\beta_Z=\left(a, a + c_1 (\ell_Z)^{-1}\right) \cup 
\left(b-c_3(\ell_Z)^{-1}, b  \right) \subset \gamma_Z = \left(a, a+ c_2 (\ell_Z)^{-1}\right) \cup 
\left(b-c_4(\ell_Z)^{-1}, b  \right) .
\]
 Therefore, we may choose a smooth function $\psi_Z$ with the properties that
\[
|Z|\|\chi_Z^{-9}\psi_{Z} \|_{\infty} \lesssim 1, \qquad \cic{1}_{\beta_Z}\leq \widehat{\psi_Z} \leq \cic{1}_{\gamma_Z}.\]
In particular,  $\widehat{\psi_Z}  =1$ on $\omega_s$ whenever $s\in \mathbb S_\subseteq^{{\boldsymbol{\omega}}}(3I)$ and $\ell_{I_s}>\ell_Z$, given that  in this case $\omega_s\subset \beta_Z$, while  $\widehat{\psi_Z} =0$ on $\omega_s$ whenever $s\in \mathbb S_\subseteq^{{\boldsymbol{\omega}}}(3I)$ and $\ell_{I_s}<\ell_Z$, given that instead $\omega_s\cap \gamma_Z =\varnothing$. 
Hence, the   main term of \eqref{e:lemma10} equals $P_{3I,\omega}*\psi_Z$, up to removal of the tiles at scale   $\ell_{I_s}=\ell_Z$, on whose frequency intervals $\widehat \psi_Z$ is not necessarily equal to zero or one. For the details,
 let $\mathbb S(Z,\omega)$ be the set of tiles with $I_s\subset 3I $, $\ell_{I_s}=\ell_Z$ and $\omega_s\in {\boldsymbol{\omega}}$.  
The spatial intervals of the tiles $\mathbb S(Z,\omega)$ are contained in $3I$, pairwise disjoint and of the same scale $\ell_Z$, so that   for $x\in Z$
\begin{equation} \label{e:lemma11}  \left|
\sum_{s\in \mathbb S(Z,\omega)} \langle f,\varphi_s
\rangle \vartheta_s(x)  \right| = \left|\sum_{s\in \mathbb S(Z,\omega)} \langle H_\omega (f\chi_{I}^9),\chi_{I}^{-9}\varphi_s
\rangle \vartheta_s(x)  \right| \lesssim \inf_{3Z} \mathrm{M}\left[ H_{\omega} (f\chi_{I}^9)\right].
\end{equation}
Then
\begin{equation}
\label{e:lemma12} \begin{split}
  \left|\sum_{\substack{s\in \mathbb S_\subseteq ^{{\boldsymbol{\omega}}} (3I)\\ \ell_{I_s}> \ell_Z }}\langle f,\varphi_s\rangle \vartheta_s\right| = \left|\left( P_{3I,\omega} f - \sum_{s\in \mathbb S(Z,\omega)} \langle f,\varphi_s\rangle \vartheta_s\right) * {\psi_Z}\right|  \lesssim \inf_{3Z} \mathrm{M}\left[ P_{3I,\omega}  f\right]+\inf_{3Z} \mathrm{M}\left[ H_{\omega} (f\chi_{I}^9)\right].
\end{split}
\end{equation}
Together with \eqref{e:retail2} and  \eqref{e:lemma11}, this completes the proof of the lemma.
\end{proof}

\subsection{Proof of Proposition \ref{p:tfasparser}}  \label{ss:pftfa}
Fix an instance of $\Omega, \Phi$ and let $f$ be a fixed compactly supported function in $L^2(\R)$. It is possible to choose $Q\in \mathcal D$ with the property that{
$
\supp f \subset (1+3^{-1})Q$.}
 A first lemma takes care of the tails
\[
P^\Phi_{\mathsf{out},\omega} f\coloneqq  P^\Phi_{\R,\omega} f  - \cic{1}_{  5Q} P^\Phi_{3Q ,\omega}f, \qquad \omega \in \Omega.
\] 
\begin{lemma}  With  $\mathcal{R}(Q) $ as in \eqref{e:ndd}, there holds
\begin{equation}
\label{e:tailhandling}
\left\|P_{\mathsf{out},\omega}^\Phi f\right\|_{\ell^2(\omega\in \Omega)} \lesssim { T_{2,\mathcal{R}(Q)}}f+\mathrm{M}_2 f
\end{equation}
pointwise almost everywhere.
\end{lemma}  

\begin{proof} In this proof, $\Phi$ is fixed and thus omitted from superscripts. First of all, note that
\[
P_{\mathsf{out},\omega} = \left[ P_{\R,\omega} f -   P_{3Q ,\omega} \right] + \cic{1}_{\mathbb R \setminus 5Q}   P_{3Q ,\omega}
\]
and the summand outside the square bracket, that is the nonlocal part of $P_{3Q,\omega}$, is immediately controlled by  (ii) of Lemma \ref{l:taildomlem}. Therefore, it suffices to control the difference $P_{\R,\omega} f -   P_{3Q ,\omega}$, which by \eqref{eq:decomp} satisfies
\[
\begin{split}
\left|P_{\R,\omega} f -   P_{3Q ,\omega}f \right|&\leq \sum_{|m|\leq 1} \sum_{k\geq 1} \left| \sum_{s\in \mathbb S ^{{\boldsymbol{\omega}}} (Q^{(k,m)})} \langle f, \varphi_s \rangle\vartheta_s \right|+ 
\sum_{|m|=2,3 } \sum_{k\geq 0} \left| P_{Q^{(k,m)},\omega}f\right|
\\
 &\eqqcolon \sum_{|m|\leq 1} \sum_{k\geq 1} U_{m,k,\omega} + 
\sum_{|m|=2,3 } \sum_{k\geq 0}  V_{m,k,\omega}.
\end{split}
\]
Applying respectively (i) and (iii) of Lemma \ref{l:taildomlem} yields for all $k\geq 0$ the pointwise estimates
\begin{equation}
\label{e:retail}
\left\|U_{m,k,\omega}\right\|_{\ell^2(\omega\in \Omega)}  \lesssim \langle f \rangle_{2,\widetilde{Q^{(k)}}} \chi_{\widetilde{Q^{(k)}}}^{9},  \qquad  \left\|V_{m,k,\omega} \right\|_{\ell^2(\omega\in \Omega)} \lesssim \langle f \rangle_{2,\widetilde{Q^{(k)}}} \chi_{\widetilde{Q^{(k)}}}^{9},
\end{equation} the first of which holds uniformly over $ k\geq 1, |m|\leq 1$, while the second holds uniformly over $ k\geq 0, |m|\in\{2,3\}$. For the second control let $\tau\in\{\pm2,\pm3\}$ and apply Lemma~\ref{l:taildomlem} (iii) with $J= Q^{(k,\tau)}$ together with the fact that $\langle f \rangle_{2,7Q^{(k,\tau)}} \simeq  \langle f \rangle_{2,\widetilde{Q^{(k)}}}$. The desired estimate follows since  
\[
\supp f\subset \widetilde{Q^{(k)}},\qquad |7Q^{(k,\tau)}|\simeq |\widetilde{Q^{(k)}}|.
\]
Fix now  $m$ and a point $x\in\mathbb R$. Summing \eqref{e:retail} up over $k$ and splitting according to whether or not $x\in\widetilde{Q^{(k)}}$ entails the claim of the lemma. 
\end{proof}

As {$ T_{2,\mathcal{R}(Q)}$ is a sparse operator} and  $\mathrm{M}_2 f$ obeys a sparse bound of the type \eqref{e:spro},   it remains to control 
\[
\cic{1}_{5Q}P^\Phi_{3Q,\omega} = \sum_{\substack{ |m|\leq 1\\ |j|\leq 2}}\cic{1}_{Q^{(0,j)}}P^\Phi_{Q^{(0,m)},\omega}. 
\]
Applying Lemma \ref{l:taildomlem} (ii) we gather that
\begin{equation}
\label{e:tailhandled} |j-m|>1 \implies
\left\|\cic{1}_{Q^{(0,j)}}P^\Phi_{Q^{(0,m)},\omega}f\right\|_{\ell^2(\omega\in \Omega)} \lesssim  \mathrm{M}_2 f.
\end{equation}
If $|j-m|\leq 1$ we instead have
\begin{equation}
	\label{e:jms}
 \cic{1}_{Q^{(0,j)}} P^\Phi_{Q^{(0,m)},\omega}= \cic{1}_{J}P^{\Phi^{j,m}}_{3J,\omega},\qquad \omega\in \Omega,
\end{equation}
having set $J=Q^{(0,j)}$ and having constructed $\Phi^{j,m}=\{ \varphi_s^{j,m}, \vartheta_s: s\in \mathbb S\}$  from $\Phi= \{ \varphi_s , \vartheta_s: s\in \mathbb S\}$ as follows: $\varphi_s^{j,m}= \varphi_s$ if $I_s\subset Q^{(0,m)}$ and $\varphi_s^{j,m}= 0$ otherwise. Applying Lemma \ref{l:mainiter} to each right hand side of \eqref{e:jms} and combining with  \eqref{e:tailhandling}--\eqref{e:tailhandled} completes the proof of Proposition \ref{p:tfasparser}.
\subsection{Rough square functions: proof of Theorem \ref{t:sparsergh}} 
Using \cite{DPFR}*{eq. (2.10)}, we learn that
\begin{equation}\label{e:smoothtf}
\left |T_\omega f\right| \lesssim \sum_{j=1} ^{285} \left| P_{\R,\omega,\boldsymbol{\omega}_j} ^{\Phi_j} f  \right|, \qquad \omega\in \Omega
\end{equation}
where, for each $j\in\{1,\ldots,285\}$, $\Phi_j$ is a suitable collection of wave packets, $\boldsymbol{\omega}_j$ is a collection of dyadic intervals satisfying properties (i)--(iii) of \S\ref{sec:tfsf}, and both $\Phi_j$ and $\boldsymbol{\omega}_j$ are constructed on the fixed shifted grid $\mathcal D^{k_j}$, for some $k_j\in \{0,1,2\}$. Proposition \ref{p:tfasparser} then immediately implies the conclusion of Theorem~\ref{t:sparsergh}.

\section{The Walsh case}\label{SWalsh} In this section, we will prove Theorem \ref{TheoremWalshRdF}. The strategy of proof involves a Walsh version of the wave packet square function $W^\Omega$. An $L^2(w)$-quantitative relation between this object and the  Walsh--Rubio de Francia square function  is provided by the Chang--Wilson--Wolff inequality, \cite{WilsonBook}*{Theorem 3.4}.

\subsection{The setting for the Walsh model} Let $\mathbb T=[0,1)$ be the $1$-torus. For $ k =0,1,2,... $, define the \emph{Walsh function $ W_{2^k} $}
\[
W_{2^k}(x) \coloneqq \sign \left( \sin \left( 2^{k+1} \pi x \right) \right),\qquad x\in \mathbb T.
\]
Now, for $ n \in \N $, write the binary expansion of $ n $ as
\[
n = \sum_{k=0}^{\infty} n_k 2^{k},
\]
where $ n_k \in \{0,1\} $ for every $ k =0,1,2,\ldots$, and define the $n$-th \emph{Walsh function} $ W_{n} $ as
\begin{equation}
	\label{e:defwf}
W_n(x) \coloneqq \prod_{k=0}^{\infty} W_{2^k}(x)^{n_k}, \qquad x\in \mathbb T.
\end{equation} The functions $\{W_n: n\in \mathbb N\}$ are the characters of the Walsh group $\big(\mathbb T,\oplus\big)$, 
where $\oplus$ stands for addition of binary digits without carry, cf.\ \cite{DG} for an introduction from the Harmonic Analysis viewpoint.
For an interval $ \omega =[k,m)\subset \mathbb R$ with $k,m \in \mathbb N$, recall the definition of Walsh frequency projection
\[
	H_{\omega} f(x) \coloneqq \sum_{n=0}^{\infty} \ind_{\omega}(n) \langle f,W_n \rangle W_n(x),\qquad x\in \mathbb T.
\]
For a collection $ \Omega =\{\omega\}_{\omega\in\Omega}$ of pairwise disjoint intervals with endpoints in $\mathbb N$, recall that the Walsh--Rubio de Francia square function is defined as
\[
T^{\Omega} f(x) \coloneqq \left( \sum_{ \omega \in \Omega } \Abs{ H_{\omega} f(x) }^2 \right)^{\frac{1}{2} } ,\qquad x\in\mathbb T.
\]
Below we describe the time-frequency model for this square function. We say that $ s = I_{s} \times \omega_{s} \subseteq \mathbb T \times [0,\infty) $ is a tile if $ I_{s} $ and $ \omega_s $ are dyadic intervals satisfying $ \abs{ I_s } \cdot \abs{ \omega_s } = 1 $. Then, for every tile $ s $ we can find an integer $ n=n(s) \in \N $ so that
\begin{equation}\label{EqWalshTileDef}
s = I_{s} \times \omega_{s} = I_s \times \frac{ 1 }{ \abs{ I_s } } [n,n+1).
\end{equation}
 Letting $\mathbb S$ denote the universe of all tiles thus defined, the notations \eqref{def:ti}, \eqref{def:somega} will be used in exactly the same way below.
Given a tile $ s \in \mathbb S$, we define the $L^2$-normalized wave packet associated to $ s $ by
\[
\varphi_s (x) \coloneqq \frac{ 1 }{ \abs{ I_s }^{\frac{1}{2}} } W_{n(s)} \left( \frac{x}{\ell_{I_s}} \right) \ind_{I_s}(x), \qquad x\in \mathbb T.
\]
  Observe that
the Haar functions arise as a special case of these wave packets by taking $ s = I \times \frac{1}{\abs{I}} [1,2) $ with $I$ dyadic subinterval of $\mathbb T$, namely
\[
h_{I}(x) \coloneqq \frac{ 1 }{ \abs{ I }^{\frac{1}{2}} } W_1 \left( \frac{x}{\ell_{I}} \right) \ind_{I}(x), \qquad x\in\mathbb T.
\]
Then, the wave packet square function for the  Walsh model is then defined by
\begin{equation}\label{eq:Walshmodel}
W^\Omega f(x) \coloneqq \left( \sum_{ s \in \mathbb {S}^\Omega  } \abs{ \langle f,\varphi_s \rangle }^2 \frac{ \ind_{ I_s }(x) }{ \abs{ I_s } } \right)^{\frac12},\qquad x\in \mathbb T.
\end{equation}
\subsection{The Walsh wave packet square function} Describing the relation between $T^\Omega$ and its wave packet model requires some preliminaries.
For  $\omega\in \Omega$, denote by $\boldsymbol{\omega}$ the collection of maximal dyadic intervals in $ \omega $. Imagining the frequency intervals as living on the vertical real axis, denote by $\boldsymbol{\omega}^{u} $ the collection of dyadic intervals $ \omega \in \boldsymbol{\omega}$ which are the  upper half of their dyadic parent,  and  by $\boldsymbol{\omega}^ d \coloneqq  \boldsymbol{\omega}\setminus \boldsymbol{\omega}^u$ those which are the lower half of their parent.
Then set
\[
\Omega^\star\coloneqq \bigcup_{\sigma\in\{u,d\}}  \bigcup_{\omega\in\Omega} \boldsymbol{\omega}^\sigma.
\]
The following lemma is a consequence of the Chang--Wilson--Wolff inequality, \cite{CWW,WilsonBook}.
\begin{lemma}\label{lem:walshcww} Let $ w \in A_{\infty,\mathcal D} $. Then, the following inequality holds:
\[
\Norm{ T^{ \Omega } f }{L^{2}(w)} \lesssim \left[w\right]_{A_{\infty,\mathcal D}}^{\frac{1}{2}}   \Norm{ W^{\Omega^\star} f }{L^{2}(w)}.
\]
\end{lemma}

\begin{proof}Let $ \omega=[k,m) $ be an interval with endpoints in $\mathbb N$, and write
\[
\begin{split}
H_\omega f 
& =\sum_{n=0}^{m-1} \langle f,W_n \rangle W_n - \sum_{n=0}^{k-1} \langle f,W_n \rangle W_n 
=\sum_{\sigma \in \{ u,d \}} \sum_{ s \in \mathbb{S}^{\boldsymbol{\omega}^\sigma }} \langle f,\varphi_{s} \rangle \varphi_{s},
\end{split}
\]
for the last identity see \cite[Section 8.1]{Thi06} (also \cite{HytonenLaceyParissis}*{p. 995}). By symmetry,  only consider the case $ \sigma = d $ and study the operator
\[
T_{d}^{\Omega} f \coloneqq \left( \sum_{\omega\in\Omega} \Abs{ \sum_{ s \in \mathbb{S}^{\boldsymbol{\omega}^d }} \langle f,\varphi_{s } \rangle \varphi_{s } }^2 \right)^{\frac{1}{2}}.
\]
Note that for each fixed $\omega\in \Omega$, the frequency components of the tiles in $\mathbb{S}^{\boldsymbol{\omega}^d }$ form a Whitney decomposition of $\omega$ with respect to the right endpoint of $\omega$. Using this fact, which in time-frequency terminology says that $\mathbb{S}^{\boldsymbol{\omega}^d }$ is a tree, together with \cite{HytonenLacey}*{Lemma 2.2} yields the identity
\begin{equation}\label{eq:treeident}
\Abs{ \sum_{ s \in \mathbb{S}^{\boldsymbol{\omega}^d }} \langle f,\varphi_{s} \rangle \varphi_{s} } =
\Abs{ \sum_{ s \in \mathbb{S}^{\boldsymbol{\omega}^d }} \langle f , W_{n(s)}  h_{I_s} \rangle h_{I_s} },
\end{equation}
where   $n(\omega)\in \mathbb N$ depends on $\omega\in\Omega$. Thus  \[
\begin{split}
& \quad \left\|\sum_{ s \in\mathbb{S}^{\boldsymbol{\omega}^d }} \langle f,\varphi_{s} \rangle \varphi_{s} \right\|_{L^2(w)}^2 
= \left\|  \sum_{ s \in \mathbb{S}^{\boldsymbol{\omega}^d }} \langle f W_{n(s)} ,  h_{I_s} \rangle h_{I_s}  \right\|_{L^2(w)}^2
\\ &\lesssim  [w]_{A_{\infty,\mathcal D}} \left\| \left(  \sum_{ s \in \mathbb{S}^{\boldsymbol{\omega}^d }} |\langle f W_{n(s)}, h_{I_s} \rangle|^2 \frac{ \ind_{I_s}}{\abs{I_s}} \right)^{\frac12} \right\|_{L^2(w)}^2    
  = [w]_{A_{\infty,\mathcal D}} \left\| \left( \sum_{ s \in  \mathbb{S}^{\boldsymbol{\omega}^d }} |\langle f , \varphi_{s} \rangle|^2 \frac{ \ind_{I_s}}{\abs{I_s}}\right)^{\frac12} \right\|_{L^2(w)}^2 
\end{split}
\]
where we have used the Chang--Wilson--Wolff inequality in the form of \cite{WilsonBook}*{Theorem 3.4} to pass to the second line, and identity \eqref{eq:treeident} again  for the last equality. Note that the right hand side of the identity above is easily seen to be bounded by $[w]_{A_{\infty,\mathcal D}} \|W^{\Omega^\star} f\|_{L^2(w)}^2$ and the proof is complete.
\end{proof}
Because of Lemma~\ref{lem:walshcww},  Theorem~\ref{TheoremWalshRdF} is reduced to  the following proposition. Here we can drop the restriction $I_s\subset \mathbb T$ in the tiles $s$ in the definition of $W^\Omega$ and just work with tiles with $I_s\subset \R$, which we implicitly assume below.
\begin{proposition}\label{prop:walshwp} Let  $\Omega$ be  a collection of pairwise disjoint intervals with endpoints in $\mathbb N$. Then{
\[
\Norm{ W^\Omega f }{ L^p(w) } \lesssim [w]_{A_{\frac p2,\mathcal D}}^{1/2} \Norm{ f }{ L^p(w) } , \qquad 2\leq p <\infty
\]
with implicit constants depending only on $p$. These bounds are sharp in terms of the power of the appearing weight characteristic.}
\end{proposition}

\begin{proof} We begin with the case $p=2$ and we prove that for any non-negative locally integrable function $w$, the following stronger estimate holds
\[
\int |W^\Omega f |^2 w \lesssim \int |f|^2  \M_\DD w ;
\]
this clearly implies the conclusion for $p=2$. Here we recall that $\M_\DD$ stands for the dyadic maximal operator. To that end, we make the qualitative assumption $w\in L^1$ which will be removed momentarily and use a layer-cake decomposition to prove an $ L^2(w) $--bound for the operator $ W^{\Omega} $. More precisely, for $ w \in A_1$ and $ f \in L^2(\mathbb T) $ say, write
\[
\begin{split}
\Norm{ W^{\Omega}  f }{ L^2(w) }^2 & = \int_{\mathbb T} \sum_{ s \in \mathbb{S}^\Omega  } \abs{ \langle f,\varphi_s \rangle }^2 \frac{ \ind_{ I_s }(x) }{ \abs{ I_s } } w(x)  \d x 
\\
& = \sum_{ s \in \mathbb{S}^\Omega } \abs{ \langle f,\varphi_s \rangle }^2 \int_{0}^{\infty} \ind_{ \left\{\frac{ w(I_s) }{ \abs{ I_s } } > \lambda\right\} }  \d \lambda = \int_{0}^{\infty} \sum_{ \substack{s \in \mathbb{S}^\Omega \\ \frac{ w(I_s) }{ \abs{ I_s } } > \lambda} } \abs{ \langle f,\varphi_s \rangle }^2  \d \lambda.
\end{split}
\]
Now, let
\[
\RR_{\lambda} \coloneqq \left\{ I_s :\, s \in \mathbb{S}^\Omega,\,\, \frac{ w(I_s) }{ \abs{ I_s } } > \lambda \right\},
\]
and denote by $ \RR_{\lambda}^{\star} $ the collection of maximal elements of $ \RR_{\lambda} $. Then,
\begin{equation}\label{EqWalshDecompMaximal}
\Norm{ W^{\Omega}   f }{ L^2(w) }^2  =   \int_{0}^{\infty} \sum_{ I^{\star} \in \RR_{\lambda }^{\star} } \sum_{ s\in \mathbb S_\subseteq^\Omega(I^\star) } \abs{ \langle f,\varphi_s \rangle }^2  \d \lambda.
\end{equation}
The rightmost term in the display above can be estimated by using the local orthogonality of the Walsh wave packets in the form
\[
\sum_{ \substack{s\in\mathbb S^\Omega_ \subseteq (I) } } \abs{ \langle f,\varphi_s \rangle }^2 \leq \int_{I} \abs{ f(x) }^2  \d x\qquad \forall I\in\mathcal D.
\]
Applying this in \eqref{EqWalshDecompMaximal} we get
\[
\Norm{ W^{\Omega}   f }{ L^2(w) }^2 \leq \int_{0}^{\infty} \sum_{ I^{\star} \in \RR_{\lambda }^{\star}} \int_{ I^{\star} } \Abs{ f(x) }^2  \d x  \d \lambda = \int_{0}^{\infty} \int_{ \bigcup_{ I^{\star} \in \RR_{\lambda }^{\star} } I^{\star} } \Abs{ f(x) }^2  \d x  \d \lambda.
\]
Recall that $\M_{\DD}$ stands for the dyadic maximal function, and observe that 
\[
\bigcup_{ I^{\star} \in \RR_{\lambda }^{\star} } I^{\star} \subset\{x:\, \M_{\DD} w(x) > \lambda\}.
\]
Thus,
\[
\Norm{ W^{\Omega} f }{ L^2(w) }^2  \leq \int_{0}^{\infty} \int_{ \M_{\DD} w(x) > \lambda } \Abs{ f(x) }^2  \d x  \d \lambda= \int_{\mathbb T} \Abs{ f(x) }^2 \M_{\DD} w(x)  \d x 
\]
which is the estimate we want to prove for $p=2$. We can now drop the assumption $w\in L^1$, for example by a monotone convergence argument, yielding the same inequality for arbitrary locally integrable non-negative functions $w$. Finally the estimate above and the definition of dyadic $A_1$ weights readily yields
\[
\Norm{ W^{\Omega} f }{ L^2(w) }^2  \leq [w]_{A_{1,\mathcal D}}\Norm{  f }{ L^2(w) }^2
\]
which is the conclusion of the proposition for $p=2$.

For $p>2$, the $L^p(w)$-estimates of the proposition follow easily by extrapolation, using for example \cite{Duo}*{Corollary 4.2}. Finally note that these bounds are sharp,  in terms of the exponents of the $A_{p,\mathcal D}$-weight. Indeed, as in \cite{LPR} any of the bounds in the conclusion of the theorem implies that the unweighted $L^p $-norms of the martingale square function grows like $p^{1/2}$, which is best-possible: any better exponent on $[w]_{A_p,\mathcal D}$  would imply a stronger, and false, $p$-growth of the martingale square function as $p\to \infty$.
\end{proof}

\subsection{Optimality in Theorem~\ref{TheoremWalshRdF}} The only thing remaining to show in order to complete the proof of Theorem~\ref{TheoremWalshRdF} is the optimality of the exponent $1$ appearing in the exponent of the weight-constant in the estimate
\[
\|T^\Omega f\|_{L^p(w)}\lesssim [w]_{A_{p/2,\mathcal D}} \|f\|_{L^p(w)},\qquad 2\leq p <\infty.
\]
For this we can just choose $\Omega$ consisting of a single interval of the form $[0,b+1)$ with $b\in\mathbb N$. We claim that an estimate of the form
\[
\sup_{b\in \mathbb N} \left\| \sum_{n\in[0,b]} \langle f,W_n\rangle W_n\right\|_{L^p(w)}\leq C_p \|f\|_{L^p(w)}
\]
implies that martingale transforms are bounded on $L^p$ with constant at most $C_p$. This observation together with the considerations in \cite{LPR} will imply again the claimed sharpness.
In order to  verify the observation above, we quote from \cite[Eq.\ (5.7)]{DPL2014}
 the equality\[
\left|\sum_{n=0} ^{b} \langle f,W_n\rangle W_n \right| =  \left|\sum_{I\in\mathcal D}\varepsilon_{I,b}\langle f W_b,h_I\rangle h_I  \right|
\]
where  $\varepsilon_{I,b}\in\{0,1\}$ is a sequence depending on $b$ only.
The right hand side is a Haar martingale transform of $f W_b$. This shows that we can recover any Haar martingale transform of $f$ by suitable choice of $b$, which is the promised claim and completes the proof of the optimality of the exponents in Theorem~\ref{TheoremWalshRdF}.

\section{Proof of Theorem \ref{thm:radial}}\label{SSufficient}

The proof of Theorem \ref{thm:radial}  begins with relating $T^\Omega$ to the  intrinsic wave packet square function $W^\Omega$ defined in  \eqref{eq:intrinsic} via a version of the Chang--Wilson--Wolff inequality.
\begin{lemma}\label{lem:CWW} For all $w\in A_\infty$ there holds
\[
\|T^\Omega\|_{L^2(w)}\lesssim [w]_{A_\infty} ^{\frac 12} \| G^\Omega \|_{L^2(w)} \lesssim  [w]_{A_\infty} ^{\frac 12} \| W^\Omega \|_{L^2(w)}.
\]
\end{lemma}

\begin{proof} The second inequality in the conclusion of the lemma is an application of \eqref{eq:intrinstowp}. The first inequality follows from an application of the Chang--Wilson--Wolff inequality \cite{CWW}, for example in the form elaborated by Lerner in \cite{LerMRL2019}*{Theorem 2.7}.
\end{proof}

The next step is to establish a sufficient condition for $L^2(w)$-boundedness of $W^\Omega$ based on the super-level sets of $\mathrm{M}w$, see \eqref{EqSufficientDesired} below. We will later show that \eqref{EqSufficientDesired} is satisfied by radially decreasing, even $A_1$ weights. Turning to the former task, let $w\in A_{1}$ and $\lambda>0$, and let $\RR_{\lambda}$ denote a collection of dyadic intervals such that $\frac{w(R)}{|R|}>\lambda$. Note that
\[
 \bigcup_{R \in \RR_{\lambda}} R \subseteq\{\M_{\DD} w>\lambda\}
\]
where $\M_{\mathcal D}$ denotes the dyadic maximal function. Arguing as in \S \ref{SWalsh} we can reduce the $ L^2(w) $ boundedness of the intrinsic wave packet square function $W^{\Omega}$ defined in  \eqref{eq:intrinsic} to	the estimate
\begin{equation}\label{EqSufficientDesired}
\int_{0}^{\infty} \sum_{\substack{t \in \mathbb S^{\Omega} \\ R_t \in \RR_{\lambda} } }|R_t|\left\langle f, \varphi_{t}\right\rangle^{2} \lesssim_w  \int|f|^{2} w,
\end{equation}
for an arbitrary choice of $\varphi_t\in \Psi_t$ for each $t\in \mathbb S^{\Omega}$. This is because
\[
\begin{split}
\Norm{ W^{\Omega}f }{ L^2(w) }^2 
&  \lesssim  \int_{\R} \sum_{ t \in \mathbb S^{\Omega}}|R_t| \abs{ \langle f,\varphi_t \rangle }^2 \frac{ \ind_{ R_t }(x) }{ \abs{ R_t } } w(x)  \d x 
\\
& = \sum_{ t \in \mathbb S^{\Omega} }|R_t| \abs{ \langle f,\varphi_t \rangle }^2 \int_{0}^{\infty} \ind_{ \left\{\frac{ w(R_t) }{ \abs{R_t} } > \lambda\right\} }  \d \lambda = \int_{0}^{\infty} \sum_{ \substack{t \in \mathbb S^{\Omega} \\R_t \in \RR_{\lambda} } }|R_t| \abs{ \langle f,\varphi_t \rangle }^2  \d \lambda.
\end{split}
\]
Thus, our goal is to prove \eqref{EqSufficientDesired}. We will use the following definition.
	
\begin{definition}Let $ \RR,\RR^\ast $ be collections of intervals. We say that  $\RR $ is \emph{subordinate} to $ \RR ^\ast $ if for every $ R \in \RR $ there exists $ R^{\ast} \in \RR^{\ast} $ such that $ R \subseteq R^{\ast} $.
\end{definition}

The canonical example of a collection $\mathcal R ^\ast$ to which $\mathcal R$ is subordinate is the collection of  its maximal elements. However, other choices are possible. Now we assume that for each $\lambda>0$ the collection $ \RR_{\lambda}$ is subordinate to $ \RR_{\lambda} ^{\ast} $. Then, in view of Lemma~\ref{l:orthol} we have the chain of inequalities
\[
\begin{split}
\int_{0}^{\infty} \sum_{R_t \in \RR_{\lambda}}|R_t|\left|\left\langle f, \varphi_{t}\right\rangle\right|^{2}&=\int_0 ^\infty \sum_{R^\ast\in\mathcal R_\lambda ^\ast}\sum_{\substack{R_t\in\mathcal R_\lambda\\R_t\subseteq R^\ast}} |R_t| \left|\left\langle f, \varphi_{t}\right\rangle\right|^{2}
\lesssim \int_{\R}|f|^{2} \int_{0}^{\infty}  \sum_{R^{\ast} \in \RR_{\lambda}^{\ast}} \chi_{R^{\ast}}^{9}.
\end{split}
\]
Thus, a sufficient condition for the desired $ L^2(w) $ boundedness \eqref{EqSufficientDesired} is that for a.e. $x\in\R$ there holds
\begin{equation}	\label{EqSufficient}
		\int_{0}^{\infty} \sum_{R^{\ast} \in \mathcal{R}_{\lambda}^{\ast}} \chi_{R^{\ast}}^{9}(x) \,\d \lambda \lesssim w(x)
\end{equation} 
where $\RR_{\lambda} ^\ast$  is such that $ \RR_{\lambda}$ is subordinate to $ \RR_{\lambda} ^\ast$ for every $\lambda>0$. By considering a single interval $R$ and taking $\lambda<w(R)/|R|$ we readily see that \eqref{EqSufficient} implies the $A_1$ condition.

\subsection{$L^2(w)$ boundedness for  even and radially decreasing $ A_1$ weights} We can show the sufficient condition \eqref{EqSufficient} for even and radially decreasing weights $ w \in A_1 $, i.e. $w(x)=w_0(|x|)$ for some $w_0:[0,\infty)\to [0,\infty)$ and $w_0$ decreasing. The proof proceeds by verifying the sufficient condition \eqref{EqSufficient}. In doing so we also provide the promised generalization of Theorem~\ref{thm:radial} of the previously known results for $w(x)\coloneqq |x|^{-\alpha}\in A_1$ to even radially decreasing $A_1$ weights on the real line.
	
Let $ R = [a,b] $ be an interval belonging to  $\RR_{\lambda}$, which we recall it is the collection of intervals such that $ \frac{w(R)}{|R|}>\lambda $. Without loss of generality, assume that $|a| < |b| $ so that $ R \subseteq [-|b|,|b|] $. Since the weight $w$ is even and decreasing we have that
\[
\begin{aligned}
\lambda<\frac{w(R)}{|R|} \leq[w]_{A_1} \inf_{x\in R} w(x) =[w]_{A_1} w_0(|b|) \qquad & \Longleftrightarrow \qquad w_0(\abs{b})>\frac{\lambda}{[w]_{A_1}} .
\end{aligned}
\]
Since $w$ is decreasing the last inequality implies the existence of some $b_\lambda=b_\lambda([w]_{A_1},w)>0$ with $w_0(b_\lambda)>\lambda/[w]_{A_1}$ such that $|b|\leq b_\lambda$.
That is, denoting $R_\lambda \coloneqq \left[ -  b_\lambda, b_\lambda\right]$ we have that $ R \subseteq R_\lambda $ and all intervals $R\in\mathcal R_\lambda$ are subordinate to the collection $\{R_\lambda\}$ for every $\lambda>0$. However,
\[
\begin{aligned}
\int_{0}^{\infty} \chi_{R_\lambda} ^9 \d \lambda \lesssim \sum_{\tau \geq 0}  2^{-18 \tau} \int_{0}^{\infty} \ind_{ \left\{\abs{x} \leq 2^{\tau} b_\lambda  \right\} }(x) \d \lambda  , 
\end{aligned}
\]
where we used the decay of $\chi_{R_\lambda} ^9$.  Observe that
\[
	|x| \leq 2^{\tau} b_\lambda \Longleftrightarrow w_0\left(\frac{|x|}{2^{\tau}}\right) \geq w_0(b_\lambda)>\frac{\lambda}{[w]_{A_1}}.
\]
	Thus,
\[
\int_{0}^{\infty} \chi_{R_\lambda} ^9 \d \lambda 	\lesssim  \sum_{\tau \geq 0} 2^{-18 \tau} \int_0^{[w]_{A_1} w_0\left(\frac{|x|}{2^{\tau}}\right) }  \d \lambda  =  \sum_{\tau \geq 0} 2^{-18\tau}[w]_{A_1} w_0\left(\frac{|x|}{2^{\tau}}\right) .
\]
Finally, note that $ \frac{|x|}{2^{\tau}} \leq|x|$, so that	
\[
\begin{aligned}
w_0\left(\frac{|x|}{2^{\tau}}\right) &=\inf_{\left(0, \frac{|x|}{2 \tau}\right)} w \leq \frac{w\left([0,|x| / 2^{\tau}]\right)}{|x| / 2^{\tau}} \leq \frac{w([0,|x|])}{|x|} \cdot 2^{\tau} 
\\
& \leq [w]_{A_1} 2^{\tau} \inf _{(0,|x|)} w =2^{\tau}[w]_{A_1} w_0(|x|).
\end{aligned}
\]
Using this in the previous inequality yields
\[
\int_{0}^{\infty} \chi_{R_\lambda} ^9 \d \lambda  \lesssim[w]_{A_1} \sum_{\tau \geq 0} 2^{-18 \tau} 2^\tau [w]_{A_1} w_0(|x|) \lesssim [w]_{A_1} ^2 w(x).
\]
This shows that even and radially decreasing $A_1$ weights satisfy the sufficient condition \eqref{EqSufficient} and thus, combined with Lemma~\ref{lem:CWW} proves Theorem~\ref{thm:radial}.

 \bibliography{SparseRDF}{}
\bibliographystyle{amsplain}
\end{document}